\documentclass[11pt,letterpaper]{amsart}
\usepackage[english]{babel}
\usepackage{amsmath,amssymb}
\usepackage{comment}
\usepackage{amsthm}
\usepackage{amscd}
\usepackage{graphicx}
\newtheorem{theorem}{Theorem}
\newtheorem{lemma}[theorem]{Lemma}

\newtheorem{question}[theorem]{Question}
\newtheorem{corollary}[theorem]{Corollary}
\newtheorem{proposition}[theorem]{Proposition}

\newtheorem*{example}{Example}
\newtheorem*{remark}{Remark}
\newtheorem*{notation}{Notation}

\usepackage{blindtext}

\title{Group rings and hyperbolic geometry}

\author{Grigori Avramidi and Thomas Delzant}

\input xy
\xyoption{all}
\begin{document}
\maketitle

\begin{abstract}
For a group acting on a hyperbolic space, we set up an algorithm in the group algebra showing that ideals generated by few elements are free, where few is a function of the minimal displacement of the action, and derive algebraic, geometric, and topological consequences.  In particular, we obtain lower bounds on Morse complexity of closed hyperbolic manifolds in terms of injectivity radius. 
\end{abstract}
\section{Introduction}

Let $G$ be a group and $\mathbb{K}$ a field. 
A natural problem is to study relations between the group $G$ and its group algebra $\mathbb K[G]$. For instance, in 1953 Fox suggested that
\begin{quote}
``{\it It seems reasonable to conjecture that a group ring $\mathbb Z[G]$ can not have divisors of zero unless $G$ has elements of finite order; this seems to be not an easy question.}''(\cite{fox}, p.557).
\end{quote} 
This was also conjectured by Kaplansky in \cite{kaplansky} and by Higman in his (unpublished) thesis (see \cite{sandling}). It is equivalent to the statement that ideals generated by one element are free modules. In the early 60's, Cohn went further (in \cite{cohnfir,cohnicm,cohnbook}) and investigated rings in which ideals generated by any number of elements are free as $\mathbb K[G]$-modules (calling them {\it free ideal rings}, or {\it firs} for short), 
and showed that group algebras of free groups have this property. Soon after, Stallings proved his celebrated result on ends of groups \cite{stallings} which implies no other group algebras do.
\begin{theorem}[Cohn \cite{cohnfir}+(a consequence of) Stallings \cite{stallings}]
\label{ffir}
The group $G$ is a free group if and only if all ideals in the group algebra $\mathbb K[G]$ are free as submodules. 
\end{theorem}
A related question is to describe, for a ring $R$, the automorphism group $\mathrm{GL}_n(R)$ of a free $R$-module. 
Denote by $\text{GE}_n(R)$ the subgroup generated by elementary and diagonal matrices. Cohn also showed that for group algebras of free groups this is the entire automorphism group.  
\begin{theorem}[Cohn \cite{cohngl}]
For every $n$, $\mathrm{GE}_n(\mathbb K[F])=\mathrm{GL}_n(\mathbb K[F])$. 
\end{theorem}
Bass \cite{bass} used these two theorems to show projective modules over the integral group ring of a free group, $\mathbb Z[F]$, are free. This algebraic result has a striking topological consequence: \begin{corollary}[3.3. in \cite{wall}]
\label{standard}
Any two-dimensional complex with free fundamental group is homotopy equivalent to a wedge of circles and $2$-spheres.
\end{corollary}

Our goal is to establish similar results for groups $G$ acting on hyperbolic spaces. Earlier steps in this direction were taken in \cite{delzant} and \cite{avramidi}, where it was proved that ideals in $\mathbb K[G]$ generated by one, respectively two elements are free if the minimum displacement of the action is large enough. Going further, we show in this paper that ideals generated by $n$ elements are free if 
\vspace{0.3cm}

\noindent ($\mathcal{H}_{n}$):\hspace{0.1cm} {\it $G$ acts on a $1$-hyperbolic space with displacement more than 
$100\log_2((n+1)!)$.}

\vspace{0.3cm}
\begin{remark}
By homogeneity, this is equivalent to the hypothesis that $G$ acts on a $\delta$-hyperbolic space with minimal displacement greater than $\delta\cdot 100\log_2((n+1)!)$.
\end{remark}
\begin{remark}
Note that $\log_2((n+1)!)=\log_2 2+\log_2 3+\dots+\log_2(n+1)$. 
Our main result is proved via an induction on $n$, adding a term of order $\log_2(k+1)$ at the $k$-th step and---in total---a correction on the order of $\log_2((n+1)!)$.
\end{remark}
Under the same hypothesis, we also describe the automorphism groups of free, rank $n$ modules.
\begin{theorem}
\label{mainfirtheorem}
Assume the group $G$ satisfies $\mathcal H_{n}$. Then 
\begin{enumerate}
\item[1.]
Every $n$-generated ideal in $\mathbb K[G]$ is a free $\mathbb K[G]$-module, and 
\item[2.]
$
\mathrm{GE}_n(\mathbb K[G])=\mathrm{GL}_n(\mathbb K[G]).
$
\end{enumerate}
\end{theorem}
\begin{example} 
If $G$ is the fundamental group of a compact riemannian manifold of curvature $\leq -1$ acting on its universal cover, then the minimal displacement is twice the injectivity radius of the manifold and $\log 2$ can serve as a hyperbolicity constant. So, if the injectivity radius is greater than $\log 2\cdot 50\log_2((n+1)!)$, then $G$ satisfies $\mathcal H_{n}$. If $G$ is any residually finite hyperbolic group, then it has a finite index subgroup satisfying $\mathcal H_{n}$. In particular, every cocompact lattice in $SO(m,1),SU(m,1)$ and $Sp(m,1)$ has a finite index subgroup that satisfies $\mathcal H_n$. The free product of any two groups satisifying ${\mathcal{H}_{n}} $ also satisfies it. 
\end{example}

\subsection*{Applications}
To augment the description of $\mathrm{GL}_n$, one needs to understand diagonal matrices, which amounts to describing the units in $\mathbb K[G]$. This was already done in \cite{delzant}, where it was shown (assuming 
minimal displacement is greater than $4\delta$) that all units are trivial, i.e. have the form $\lambda g$ for $\lambda\in\mathbb K^*,g\in G$. In particular, for a finite field $\mathbb K$ and finitely generated group $G$ the unit group of $\mathbb K[G]$ is finitely generated. Combining this with our theorem, we obtain an analogous result for $\mathrm{GL}_n$.
\begin{theorem}
\label{fg}
Assume $G$ satisfies $\mathcal H_{n}$. If the field $\mathbb K$ is finite and the group $G$ is finitely generated, then the group $\mathrm{GL}_n(\mathbb K[G])$ is finitely generated. 
\end{theorem}
In a different direction, a geometric consequence of our theorem is a lower bound for critical points of Morse functions of a given index on essential manifolds.
\begin{theorem}
\label{manifold}
Let $X^d$ be a closed $d$-manifold, $G$ a group that satisfies $\mathcal H_{n}$ and $BG$ its classifying space. If there is a continuous map $f:X^d\rightarrow BG$ with $f_*[X]\not=0$ in $H_d(BG;\mathbb K)$ then for each $0<k<d$, a Morse function on $X^d$ has at least $n+1$ critical points of index $k$.  
\end{theorem}
The theorem applies---for instance---if $X=BG$ is a closed, riemannian manifold of curvature $\leq -1$ and injectivity radius greater than $50\log((n+1)!)$, or more generally if $X$ has a map of non-zero degree to such a $BG$.

Using Bass's local-to-global method, we also obtain the following version of Corollary \ref{standard} for some $2$-dimensional hyperbolic groups (e.g. very high genus surface groups).
\begin{theorem}
\label{hypstandard}
Assume that $G$ satisfies $\mathcal H_{n}$, and that there is an aspherical $2$-complex $Y$ with fundamental group $G$. Then every presentation $2$-complex for $G$ with less than $n+1$ relations is homotopy equivalent to 
$$
Y\vee S^2\vee\dots\vee S^2.
$$
\end{theorem}
It is known (\cite{cohenlyndon}) that one-relator groups have geometric dimenension at most two. In our setting of groups acting with large displacement on hyperbolic spaces, we can show that ``few-relator'' groups have cohomological dimension $\leq 2$. 
\begin{theorem}
\label{fewrelator}
An $n$-relator group satisfying $\mathcal H_{n}$ has cohomological dimension $\leq 2$. 
\end{theorem}
Finally, let us mention a consequence that can be thought of either as a generalization of Theorem \ref{fewrelator} to higher dimensions or as a generalization of (the $X=BG$ case of\footnote{See also Theorem \ref{essential}.}) Theorem \ref{manifold} to hyperbolic groups.
\begin{theorem}
\label{cellbound}
Assume that the group $G$ satisfies $\mathcal H_{n}$ and has cohomological dimension $d$. Then, every aspherical complex with fundamental group $G$ has more than $n$ cells in each dimension $0<k<d$. 
\end{theorem}

\subsection*{Algorithms}
The proof of the first part of Theorem \ref{mainfirtheorem} is based on an algorithm which can be seen as a geometric version of the euclidean algorithm. 
To describe it, let us first sketch the approach to the ``only if'' direction of Theorem \ref{ffir} given in Cohn's book \cite{cohnbook}. It consists of two distinct steps: 
\begin{itemize}
\item
Fix a basis for the free group $F$, let $F_+$ be the (monoid of) non-negative words in this basis and denote by $\mathbb K[F_+]\subset\mathbb K[F]$ the subring of the group algebra generated by these words. Cohn first shows that ideals in $\mathbb K[F_+]$ are free. (Corollary 2.5.2 and Theorem 2.4.6) 
\item
He then gives a localization procedure for passing from $\mathbb K[F_+]$ to $\mathbb K[F]$ that preserves the free ideal property. (Corollary 7.11.8)
\end{itemize}

Cohn measures size via filtrations. Recall that a {{\it filtration}} on a ring $R$ is a map $|\cdot|
: R\rightarrow \mathbb{N}\cup\{-\infty\}$  such that $|0|=-\infty$, $|1| = 0$, and for all $x,y\in R$
$
|x-y|\leqslant\max(|x|,|y|),
$
$
|xy|\leqslant|x| + |y|.
$

Given a generating set for a group $G$, the group algebra $\mathbb K[G]$ has a natural filtration by word length $l(g)$ of elements $g$ in $G$:
$$
\left|\sum_{g\in G}\lambda^g\cdot g\right|:=\max_{\lambda^g\not=0}l(g).
$$
On the ring $\mathbb K[F_+]$ of positive words in the free group, this filtration satisfies $|xy|=|x|+|y|$. 

The relevant notion of dependence says that some linear combination has smaller size than the maximum dictated by its terms: A family $\xi_1, \dots, \xi_n$ in $R$ is {\it $|\cdot|$-dependent} if there is a non-zero $(\alpha_1, \ldots, \alpha_n)\in R^n$ such that 
$$
\left|\sum_i \alpha_i \xi_i\right|<\max_i\left|\alpha_i\xi_i\right|.
$$
\begin{example}
This is the case if the family is linearly dependent in the usual sense.
\end{example}
The first step is accomplished via the following algorithmic theorem.
\begin{theorem}[Cohn \cite{cohnbook}]
Let $|\cdot|$ be the word length filtration of $\mathbb K[F_+]$. If $\xi_1,\dots,\xi_n$ in $\mathbb K[F_+]$ is a $|\cdot|$-dependent family then, up to reordering, there exist $\beta_2,\dots,\beta_n$ in $\mathbb K[F_+]$ such that 
$$
\left|\xi_1 + \sum_{i = 2}^n \beta_i \xi_i\right| < | \xi_1|.
$$
\end{theorem}
For a finitely generated ideal in $\mathbb K[F_+]$, one can repeatedly apply this theorem to decrease the size of members of a finite generating set until one arrives at a generating set that does not satisfy any dependence relations and hence forms a basis, showing the ideal is free. Doing a bit more work, Cohn uses this theorem to show infinitely generated ideals in $\mathbb K[F_+]$ are free, as well (Thm. 2.4.6). Finally, Cohn's localization proceedure shows the same is true for $\mathbb K[F]$. 


In \cite{hogangeloni}, Hog-Angeloni gave a beautiful, geometric version of this algorithm that applies to the entire group algera $\mathbb K[F]$. It can be used to bypass the localization step if one is interested exclusively in finitely generated ideals. Hog-Angeloni's proof goes by looking at the action of $F$ on its Cayley graph, which is a tree $\mathcal T$. It uses the same notion of dependence as Cohn's proof 
but the natural notion of size of group ring elements relevant to her argument is the {\it diameter} 
$$
\mathrm{diam}\left(\sum_{g\in G}\lambda^g\cdot g\right):=\max_{\lambda^g\not=0\not=\lambda^h}l(g^{-1}h). 
$$
\begin{remark}
Note that $|\xi|=0$ means $\xi \in\mathbb{K}^*$, while $\mathrm{diam} (\xi) = 0$ means $\xi=\lambda g$ for a non-zero $\lambda\in \mathbb{K}^*$ and group element $g \in F$. The diameter is invariant by left translation, while $|\cdot|$ is not. 
\end{remark}


\begin{theorem}[Hog-Angeloni \cite{hogangeloni}] 
Let $|\cdot|$ be the word length filtration of $\mathbb K[F]$. 
If $\xi_1,\dots,\xi_n$ in $\mathbb K[F]$ is a $|\cdot|$-dependent family then, up to reordering, there exist $\beta_2,\dots,\beta_n$ in $\mathbb K[F]$ such that
$$
\mathrm{diam}\left(\xi_1+\sum_{i=2}^n\beta_i\xi_i\right)<\mathrm{diam}(\xi_1).
$$
\end{theorem}

An attentive reader of \cite{hogangeloni} can check that Hog-Angeloni does not use that the group is free nor that the tree is a Cayley graph, but only the fact that the group acts freely on an $\mathbb R$-tree. We build on her geometric approach, replacing $F$ acting on the tree $\mathcal T$ by a group $G$ that acts on a hyperbolic space $\mathcal H$ with large minimum displacement. 
In our approach we will use the natural filtration on $\mathbb K[G]$ obtained from the action of $G$ on the hyperbolic space $\mathcal H$. 

\begin{notation} 
Let $\mathcal H$ be a geodesic metric space and $o$ a basepoint (or origin). Denote by $|p-q|$ the distance between two points in $\mathcal H$, and by $|p|=|p-o|$ the distance to the origin. If $\mathcal X\subset\mathcal H$ is a finite subset then $|\mathcal X|=\max_{p\in\mathcal H}|p|$ is called its absolute value and $\mathrm{diam}(\mathcal X)=\max_{p,q\in\mathcal X}|p-q|$ its diameter. 

If a group $G$ acts isometrically on $\mathcal H$ and $\xi=\sum_{g\in G}\lambda^g\cdot g$ is an element in the group algebra, denote by $\mathcal X=\{g\cdot o\mid \lambda^g\not=0\}$ the orbit of the basepoint under group elements appearing with non-zero coefficient in $\xi$ (i.e. under group elements in the algebraic support of $\xi$) and call it the geometric support of $\xi$.
The diameter $\mathrm{diam}(\xi)$ and absolute value $|\xi|$ are $\mathrm{diam}(\xi)=\mathrm{diam}(\mathcal X)$ and $|\xi|=|\mathcal X|$. By convention, $\mathrm{diam} (0) = | 0 | = - \infty$. 

The minimal displacement of the action of $G$ on $\mathcal H$ is $\inf_{g\in G-\{1\}, p\in\mathcal H}|g\cdot p-p|$. 
\end{notation}

Now, we can state our hyperbolic version of Hog-Angeloni's theorem. 
\begin{theorem}
\label{intromaintheorem}
Set $\delta_n=(15n+2\log_2((n+1)!))\delta$. Let the group $G$ act on a $\delta$-hyperbolic space $\mathcal H$ with minimum displacement $>4\delta_n+(10+2n)\delta$. If $\xi_1,\dots,\xi_n\in\mathbb K[G]$ and there is a non-zero $(\alpha_1,\dots,\alpha_n)\in\mathbb K[G]^n$ such that 
$$
\left|\sum_i\alpha_i\xi_i\right|<\max_i|\alpha_i\xi_i|-\delta_n, 
$$
then, up to re-ordering, there exist $\beta_2,\dots,\beta_n$ in $\mathbb K[G]$ such that
$$
\mathrm{diam} \left(\xi_1+\sum_{i = 2}^n \beta_i \xi_i\right)<\mathrm{diam}(\xi_1)-\delta.
$$
\end{theorem}
\begin{remark}
Theorem \ref{intromaintheorem} will be used to replace the euclidean algorithm in the classical study of ideals, submodules of free modules, and matrices over euclidean rings. This is the key tool needed to obtain the geometric and algebraic applications (Theorems \ref{mainfirtheorem} through \ref{cellbound}) as we shall see in the last section. 
\end{remark}

\subsection*{Acknowledgements}
We would like to thank Misha Gromov for his interest in this work and for mentioning the possibility of extending our results to essential manifolds. G.A. would like to thank the Max Planck Institut f\"ur Mathematik for its hospitality and financial support. We would like to thank the referee for pointing out a gap in an earlier version of the proof of Theorem 40, for their useful comments, and for highlighting an interesting perspective concerning the structure of the ring $\mathbb K[\mathbb Z\times G]$ (see \ref{zsubsection}).

\section{Hyperbolic preliminaries\label{hypsection}}
Fix a $\delta$-hyperbolic metric space $\mathcal H$ with basepoint $o$. See \cite{gromovhyperbolic},\cite{cdp}, or \cite{bhbook} for background on hyperbolicity.

\subsubsection*{Radii and centers}
Let $\mathcal X$ be a bounded subset in the metric space $\mathcal H$. Denote by 
$$
r(\mathcal X):=\inf\{r\mid\mbox{there is } c \mbox{ for which } B(c,r)\supset\mathcal X\}
$$ 
the infimum of radii of closed balls containing $\mathcal X$. We call it the {\it radius} of $\mathcal X$. If $\mathcal H$ is a proper metric space (i.e. if its closed, bounded sets are compact), then this infimum is realized and there is a closed ball of radius $r(\mathcal X)$ containing $\mathcal X$. If $\mathcal H$ is a complete, $\mathrm{CAT}(0)$ space, then there is a unique such ball. In general, we only have for any positive $\epsilon$ a closed ball $B(c,r(\mathcal X)+\epsilon)$ containing $\mathcal X$. We call such a $c$ an $\epsilon$-center of $\mathcal X$.


If $\mathcal H$ is an arbitrary metric space, one can embed $\mathcal H$ isometrically in a space in which every finite subset has a $0$-center. Let $\omega$ be a non-principal ultrafilter; the ultralimit $\mathcal H^\omega$ of the constant sequence $(\mathcal H, x_0, d)$ satisfies this property, as can be easily checked. Furthermore, if $\mathcal H$ is geodesic and $\delta$-hyperbolic then the ultralimit $\mathcal H^{\omega}$ is, as well, and every isometry of $\mathcal H$ extends to an isometry of $H^\omega$ with the same infimum displacement.

\subsubsection*{Gromov products}
Recall (\cite{gromovhyperbolic,cdp,bhbook}) the definition of the Gromov product
$$
\left<p,q\right>_r:={1\over 2}(|p-r|+|q-r|-|p-q|).
$$
To simplify notation, recall that the distance from $p$ to the origin $o$ is denoted by $|p|$ and set $\left<p,q\right>:=\left<p,q\right>_o$. First, we give an estimate for the Gromov product of an $\epsilon$-center of $\mathcal X$ with a point of $\mathcal X$ that follows directly from the triangle inequality.
\begin{lemma}
\label{fellowtravel1}
Let $p$ be a point in a set $\mathcal X$ with $\epsilon$-center $c$ and radius $r$. Then
$$
|c|\geq\left<p,c\right>\geq{|\mathcal X|+|p|\over 2}-r-\epsilon.
$$
\end{lemma}
\begin{proof}
The triangle inequality $|c|\geq|p|-|c-p|$ implies
\begin{eqnarray*}
|c|&\geq&{1\over 2}\left(|c|+|p|-|c-p|\right)\\
&\geq&|p|-|c-p|\\
&\geq&|p|-r-\epsilon.
\end{eqnarray*}
This proves the left inequality and, since it is true for all $p\in\mathcal X$, shows $|c|\geq|\mathcal X|-r-\epsilon$. Using this inequality to bound the Gromov product we get:
\begin{eqnarray*}
\left<p,c\right>&=&{1\over 2}(|c|+|p|-|p-c|)\\
&\geq&{1\over 2}((|\mathcal X|-r-\epsilon)+|p|-(r+\epsilon)).
\end{eqnarray*}
\end{proof}

\subsubsection*{Thin triangles and projections in hyperbolic spaces}
\begin{figure}[h!]
\label{thintriangle}
\centering
\includegraphics[scale=0.2]{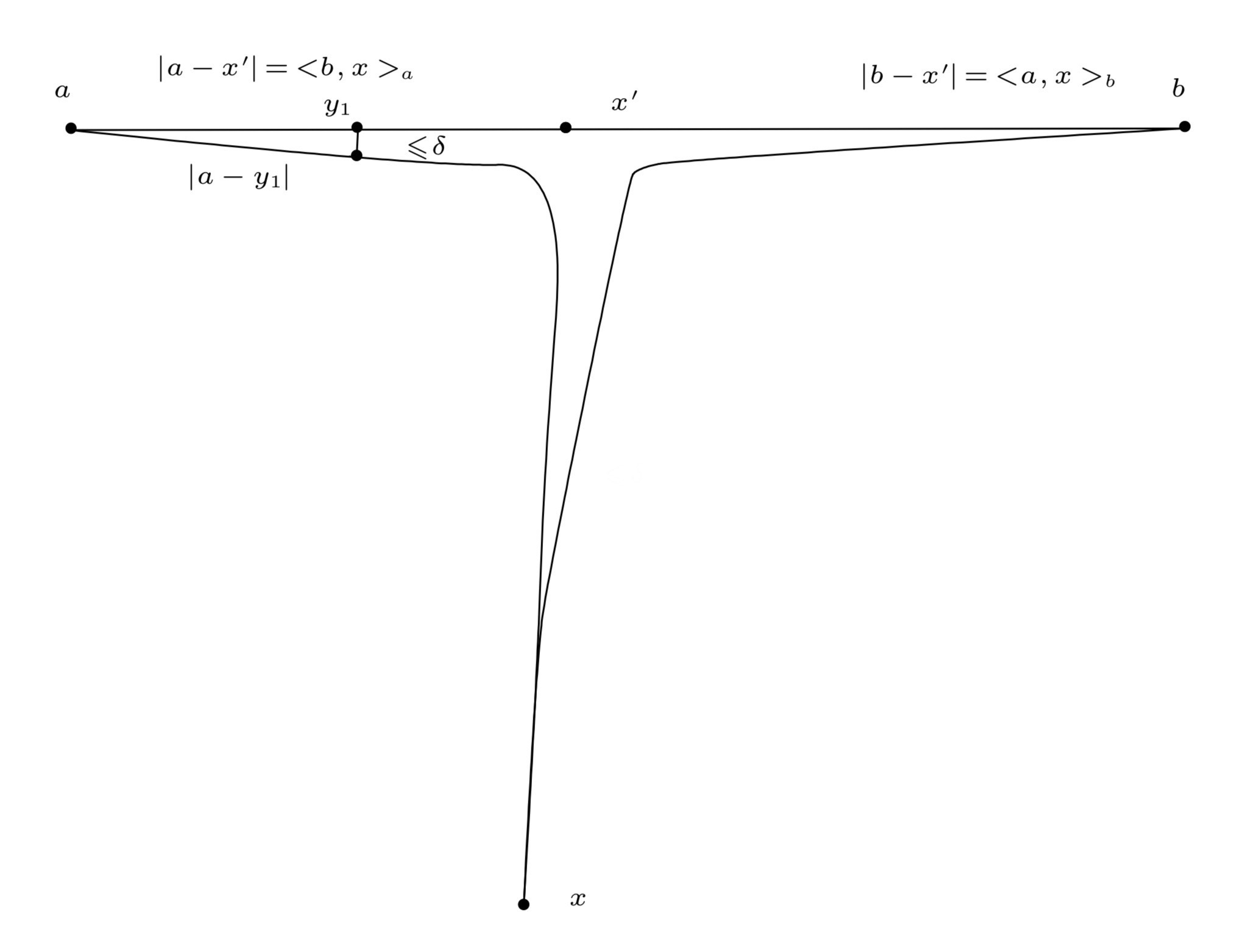}
\caption{A thin triangle}
\end{figure}
In a $\delta$-hyperbolic space, any geodesic triangle is $\delta$-thin. This means that if $[r,p,q]$ is a triangle then two (oriented) geodesics $[r,p]$ and $[r,q]$ parametrized by arc length $p(t),q(t)$ remain $\delta$-close (i.e. $|p(t)-q(t)|\leq\delta$) until $t=\left<p,q\right>_r$.
We will repeatedly use the following consequence of this: Given a segment $[a,b]$ and a point $x\in\mathcal H$, the {\it projection of $x$ to $[a,b]$} is the point $x'$ on $[a,b]$ such that $|a-x'|=\left<b,x\right>_a$.\footnote{This is also the point such that $|b-x'|=\left<a,x\right>_b$, since $\left<a,x\right>_b+\left<b,x\right>_a=|a-b|$.} Then, for points $y_1$ in the initial subsegment $[a,x']$ of $[a,b]$ we have a $\delta$-converse to the triangle inequality (see Figure 1):
\begin{equation}
\label{conversetriangle}
|x-y_1|+|y_1-a|\leq|x-a|+\delta,
\end{equation}
while for points $y_2$ in the subsegment $[x',b]$ we have 
$$
|x-y_2|+|y_2-b|\leq|x-b|+\delta.
$$

\subsubsection*{Diameter vs radius}
First, we note that in a $\delta$-hyperbolic space---just like in a tree---the diameter and radius are closely related. More precisely, 
\begin{lemma}
\label{diamvrad}
Let $[a,b]$ be a diameter-realizing segment of a finite set $\mathcal X$ with radius $r$ and let $m$ be the midpoint of $[a,b]$. Then $\mathcal X\subset B\left(m,{|a-b|\over 2}+\delta\right)$. In particular, $m$ is a $\delta$-center of $\mathcal X$ and
$$
{|a-b|\over 2}\leq r\leq {|a-b|\over 2}+\delta.
$$
\end{lemma}
\begin{proof}
The left inequality holds in any geodesic space. For the right one, let $p$ be a point in  $\mathcal X$ and assume (without loss of generality) that its projection to $[a,b]$ is on $[m,b]$. Then
\begin{eqnarray*}
{|a-b|\over 2}+|m-p|&=&|a-m|+|m-p|\\
&\leq&|a-p|+\delta\\
&\leq&|a-b|+\delta,
\end{eqnarray*}
whence the result.
\end{proof}
\subsubsection*{Center vs midpoint}
Next, we observe that an $\epsilon$-center is $(2\delta+\epsilon)$-close to the midpoint of any diameter realizing segment.
\begin{lemma}
\label{centervsmidpoint}
Let $m$ be the midpoint of a diameter $[a,b]$ of a finite set $\mathcal X$ with $\epsilon$-center $c$. Then 
$$
|c-m|\leq \epsilon+2\delta.
$$  
\end{lemma}
\begin{proof}
Without loss of generality, the projection of $c$ to $[a,b]$ is on $[m,b]$, so
\begin{eqnarray*}
|c-m|+{|a-b|\over 2}&=&|c-m|+|m-a|\\
&\leq&|c-a|+\delta\\
&\leq&r+\epsilon+\delta
\end{eqnarray*} 
implies $|c-m|\leq r-{|a-b|\over 2}+\epsilon+\delta\leq\epsilon+ 2\delta$. 
\end{proof}
In particular, any two $\epsilon$-centers of a finite set are $(4\delta+2\epsilon)$-close to each other and---by Lemma \ref{diamvrad}---the midpoints of any two diameters are $3\delta$-close to each other.

\subsubsection*{An equivariant choice of centers}
\begin{corollary}
\label{equivcenter}
Suppose $G$ acts on a $\delta$-hyperbolic space $\mathcal H$ with displacement greater than $3\delta$. Then $G$ acts freely on the collection of all finite subsets of $\mathcal H$. 
\end{corollary}
\begin{proof}
Let $\mathcal X$ be a finite subset and $m$ the midpoint of a diameter of $\mathcal X$. If $g\mathcal X=\mathcal X$ for some $g\in G$ then $m$ and $gm$ are $\delta$-centers by Lemma \ref{diamvrad} and hence $3\delta$-close by Lemma \ref{centervsmidpoint}. Therefore, $g=1$. 
\end{proof}
Therefore, as long as the displacement is greater than $3\delta$ we can choose for each finite subset $\mathcal X\subset \mathcal H$ an $\epsilon$-center $c(\mathcal X)$ such that $c(g\mathcal X)=gc(\mathcal X)$ for each $g\in G$.

\subsubsection*{Distance from origin to center}
Now, we estimate the distance from the origin to an $\epsilon$-center of a set $\mathcal X$ in terms of its radius and $|\mathcal X|$.
\begin{lemma}
\label{distancetocenter}
Let $\mathcal X$ be a finite set with $\epsilon$-center $c$ and radius $r$. Then
$$
|\mathcal X|-r-\epsilon\leq|c|\leq |\mathcal X|-r+4\delta+\epsilon
$$
\end{lemma}
\begin{proof}
The left inequality already appeared in Lemma \ref{fellowtravel1}. For the right one, we again pick a diameter realizing segment $[a,b]$ and let $m$ be its midpoint. 
Without loss of generality, the projection of $o$ to $[a,b]$ is on $[m,b]$ and we have 
\begin{eqnarray*}
|c|&\leq&|m|+|m-c|\\
&\leq&(|a|-|a-m|+\delta)+2\delta+\epsilon\\
&\leq& |\mathcal X|-{|a-b|\over 2}+3\delta+\epsilon,
\end{eqnarray*}
where the first inequality is the triangle inequality, the second is inequality (\ref{conversetriangle}) for $x=o$ and $y_1=m$ (it reads $|o-m|+|m-a|\leq|o-a|+\delta$) together with Lemma \ref{centervsmidpoint}, and the third follows from the fact that $a\in\mathcal X$ and $|a-m|=|a-b|/2$ since $m$ is the midpoint.
Putting this together with Lemma \ref{diamvrad} gives
$$
|c|+r\leq|c|+\left({|a-b|\over 2}+\delta\right)\leq |\mathcal X|+4\delta+\epsilon
$$
which proves the desired inequality. 
\end{proof}

\subsubsection*{Diameter of intersection of two balls}
Another property of $\delta$-hyperbolic spaces we will need is a bound on the diameter of the intersection of two balls. The diameter of a set $\mathcal X$ (not necessarily finite) is $\inf_{a,b\in\mathcal X}|a-b|$. 
\begin{lemma}
\label{twoballs}
In a $\delta$-hyperbolic space $\mathcal H$, 
the diameter of the intersection of two closed balls $B(c_1,r_1)\cap B(c_2,r_2)$ is bounded above by 
$
r_1+r_2-|c_1-c_2|+2\delta. 
$
\end{lemma}
\begin{proof}
Let $D$ be the diameter of the intersection $B(c_1,r_1)\cap B(c_2,r_2)$. For $\alpha>0$, pick points $a$ and $b$ in this intersection such that $|a-b|\geq D-\alpha$. Let $a',b'$ be the two projections of $a, b$ on $[c_1, c_2]$. We assume that $a'$ is on the left of $b'$ on this segment. Using $a'\in[c_1,b']$, hyperbolicity and the fact that $b \in B (c_1, r_1)$
\begin{eqnarray*}
| b - b' | + |b'-a'|+|a'-c_1|&=&|b-b'|+|b'-c_1|\\
&\leqslant&\delta+| b - c_1| \\
&\leqslant&\delta+r_1.
\end{eqnarray*}
Similarly, $b'\in[a',c_2]$, hyperbolicity and the fact that $a\in B(c_2,r_2)$ implies
\begin{eqnarray*}
|a-a'|+|a'-b'|+|b'-c_2|&=&| a - a' | + | a' - c_2 | \\
&\leqslant& \delta + | a - c_2|\\
&\leqslant& \delta+r_2.
\end{eqnarray*}
Adding these two inequalities together and using the triangle inequality on the left
gives
$$
|a-b|+|c_1-c_2|\leq 2\delta+r_1+r_2. 
$$
Since $|a-b|\geq D-\alpha$ and the argument applies to every $\alpha>0$, this finishes the proof. 
\end{proof}
\subsubsection*{Gromov product inequality} Finally, we recall a key inequality that will be used to estimate distances: 
\begin{lemma}[\cite{gromovhyperbolic} \&6 page 155 or \cite{cdp} 8.2 page 91]
\label{gromovproductinequality}
For a sequence of $2^k+1$ points $x_0,\dots,x_{2^k}$, we have 
$$
\left<x_0,x_{2^k}\right>\geq\min_i\left<x_i,x_{i+1}\right>-k\delta. 
$$
\end{lemma}
\begin{remark}
The same inequality holds for any sequence of fewer than $2^k+1$ points, as we can extend it to a sequence of $2^k+1$ points by repeating the point in the sequence which has maximal distance to $o$. 
\end{remark}

\section{Extremal graphs\label{extremalsection}}

Let $V$ be a finite set and suppose we have a family\footnote{Repetitions in $Y=(\mathcal Y_v)_{v\in V}$ are allowed, i.e. we may have $\mathcal Y_v=\mathcal Y_w$ even if $v\not=w$.} $Y=(\mathcal Y_v)_{v\in V}$ of bounded subsets of $\mathcal H$. 
For subsection \ref{extremalgraphsubsection}, we fix $\epsilon$ arbitrarily small (say $\epsilon<<\delta$) and chose for every set $\mathcal Y_v$ an $\epsilon$-center $c_v$.
(If $\mathcal H$ is proper or complete CAT(0), we set $\epsilon=0$ and let $c_v$ be a center of $\mathcal Y_v$.) Set $d=|\cup_{v\in V}\mathcal Y_v|=\max_{v\in V}|\mathcal Y_v|$.
Let $\mu$ be a fixed parameter. A point $p\in\cup_{v\in V}\mathcal Y_v$ satisfying $|p|\geq |\cup_{v\in V}\mathcal Y_v|-\mu$ is called a {\it $\mu$-extremal point} (of $Y$). A $0$-extremal point will also be called an {\it extremal point}. 

\subsection{The graph $\Gamma_{\mu}(Y)$\label{extremalgraphsubsection}}
Let us first define the {\it $\mu$-extremal graph $\Gamma_{\mu}(Y)$} of the family $Y$. The vertices of $\Gamma_{\mu}(Y)$ are the indices $v\in V$ such that $\mathcal Y_v$ contains a $\mu$-extremal point (of $Y$), and there is an edge between $v$ and $w$ for each $\mu$-extremal point in $\mathcal Y_v\cap\mathcal Y_w$.



\subsubsection*{Distance between centers} 
For each vertex $v\in V$, let $r_v$ be radius of the set $\mathcal Y_v$. 
We estimate the distance between two centers $c_v$ and $c_w$ in terms of the distance between the vertices $v$ and $w$ in the graph $\Gamma_{\mu}(Y)$. 
By a {\it (simplicial) path} between $v$ and $w$, we mean a sequence $v=v_0,v_1,\dots,v_m=w$ of vertices, such that for each $i$, $v_i$ and $v_{i+1}$ are connected by an edge $e_i$ in the graph. 
\begin{lemma}
\label{centerdistance}
For a subset $W\subset V$ denote $r_{\max(W)}=\max_{v\in W}r_v$. 
\begin{enumerate}
\item[1.]
If $v$ and $w$ are connected by a path $P$ of length $m$ in $\Gamma_\mu(Y)$, then 
$$
|c_v-c_w|\leqslant (r_{\max(P)}-r_v)+(r_{\max(P)}-r_w)+2\mu+(8+2\lceil\log_2(2m)\rceil)\delta+4\epsilon.
$$
\item[2.]
If $v$ and $w$ are adjacent vertices in $\Gamma_{\mu}(Y)$, then we have
$$
|c_v-c_w|\leq|r_v-r_w|+2\mu+10\delta+4\epsilon.
$$
\end{enumerate}
\end{lemma}

\begin{proof}
The path from $v$ to $w$ defines an alternating sequence of vertices and edges $v=v_0,e_0,v_1,e_1,\dots,e_{m-1},v_m=w$ in the graph, which gives an alternating sequence of centers and $\mu$-extremal points $c_{v_0},p_{e_0},\dots,p_{e_{m-1}},c_{v_m}$.
Lemma \ref{fellowtravel1} implies for $v=v_i$ and $p\in\{p_{e_{i-1}},p_{e_i}\}$ that 
\begin{eqnarray*}
\left<c_{v_i},p\right>&\geq&{|\mathcal Y_{v_i}|+|p|\over 2}-r_{v_i}-\epsilon,\\
&\geq& d-\mu-r_{v_i}-\epsilon
\end{eqnarray*}
since $|\mathcal Y_{v_i}|\geq d-\mu$ and $|p|\geq d-\mu$. 
The sequence $c_{v_0},p_{e_0},\dots,c_{v_m}$ consists of $2m$ points and $2m\leq 2^{\lceil\log_2(2m)\rceil}+1$, so the remark after Lemma \ref{gromovproductinequality} implies 
\begin{eqnarray*}
\left<c_v,c_w\right>&\geq&\min_{v_i} (d-\mu-r_{v_i}-\epsilon)-\lceil\log_2(2m)\rceil\delta.
\end{eqnarray*}
Using this and the inequality $|c_v|\leq d-r_v+4\delta+\epsilon$ obtained in Lemma \ref{distancetocenter} we get
$$
|c_v|-\left<c_v,c_w\right>\leq (\max_{v_i}r_{v_i})-r_v+\mu+(4+\lceil\log_2(2m)\rceil)\delta+2\epsilon,
$$
and a similar inequality for $w$ in place of $v$. Since
$$
|c_v-c_w|=(|c_v|-\left<c_v,c_w\right>)+(|c_w|-\left<c_v,c_w\right>),
$$
we obtain the first inequality. For adjacent vertices $v$ and $w$, we have $m=1$ and $2\max\{r_v,r_w\}-r_v-r_w=|r_v-r_w|$, which gives the second inequality. 
\end{proof}

\subsubsection*{Colors} 
Consider a partition $V=V_1\sqcup\dots\sqcup V_n$ into $n$ subsets called {\it colors} such that any two vertices of the same color have the same radius (if $v,w\in V_i$ then $r_v=r_w$).  Denote by $r_i$ the radius of a set of color $i$. 

\subsubsection*{Bounding the diameter of $\Gamma_{\mu}(Y)$}
Next, we will give an upper bound on the diameter of a component of an extremal graph in terms of the number $n$ of colors. The result will be proved by induction on the number of colors $n$. In fact, we prove something stronger, namely an upper bound on the length of an embedded path in the graph. 
\begin{proposition}
\label{pathbound}
Let $V=V_1\sqcup\dots\sqcup V_n$ be a finite set partitioned into $n$ colors and let $Y=(\mathcal Y_v)_{v\in V}$ be a collection of bounded subsets of $\mathcal H$ such that for each $i$ and any two different $v,w\in V_i$ we have
\begin{itemize}
\item
$r_v=r_w$ and 
\item
$|c_v-c_w|>2\mu+(10+2n)\delta+4\epsilon$. 
\end{itemize}
Then every embedded path in $\Gamma_{\mu}(Y)$ has length at most $2^n-2$. 
\end{proposition}
\begin{proof}
We argue by contradiction. If there is an embedded path in $\Gamma_{\mu}(Y)$ of length more than $2^n-2$, then there is an embedded path of length precisely $2^n-1$. Let $P_n$ be such a path. It has $2^n$ vertices. Reorder the list of colors so that
$$
r_1\geq\dots\geq r_n. 
$$
Let $P_k$ be the sequence of vertices obtained by throwing out from $P_n$ all vertices of colors $\{k+1,\dots,n\}$. The sequence $P_k$ is colored by the set $\{1,\dots,k\}$. We prove by reverse induction (starting with base case $k=n$ and going down to $k=1$) that 
\begin{itemize}
\item
$P_k$ has at least $2^k$ vertices, and 
\item
any two consecutive vertices in $P_k$ have different colors.
\end{itemize}
Note that consecutive vertices of $P_n$ have different colors by Lemma \ref{centerdistance}.2 and the second hypothesis. So $P_n$ satisfies the two properties above, verifying the base case. 

To prove the induction step, suppose we know the statement for $P_{k+1}$. Since it has at least $2^{k+1}$ vertices an no two consecutive vertices have the same color, at most $\lceil|P_{k+1}|/2\rceil$ have the color $r_{k+1}$, so that $P_k$ has at least $2^k$ vertices. Suppose two consecutive vertices $v,w$ in the sequence $P_k$ have the same color. Then $r_v=r_w$ and the two vertices are connected by a path (of length at most $2^n$) in $P_n$ through colors $\{k+1,\dots,n\}$ in which all vertices have radii $\leq r_v=r_w$ by our choice of ordering. Therefore, Lemma \ref{centerdistance}.1 implies: 
$$
|c_v-c_w|\leq 2\mu+(8+2(n+1))\delta+4\epsilon.
$$
This contradicts the second hypothesis of our proposition. So, we have shown that $v$ and $w$ have different colors, completing the induction step. 

We have shown $P_1$ has at least two vertices and consecutive ones have different colors. This is absurd since all vertices of $P_1$ have the same color. So, we arrive at a contradiction to our initial assumption, and conclude that all embedded paths in $\Gamma_{\mu}(Y)$ have length $\leq 2^n-2$. 
\end{proof}
\begin{remark}
As a consequence of our argument, we see that the set $P_1$ has at most one element, as it is colored by a unique color, and an embedded path in $\Gamma_{\mu}$ contains at most one element with the largest color.
\end{remark}
Recall that $r_1\geq\dots\geq r_n$ so that $r_{\max(V)}=r_1$.
\begin{corollary}[Uniqueness]
\label{uniqueness}
In the situation of Prop. \ref{pathbound}, if for any pair of distinct vertices of the same color $v,w\in V_i$ we have 
$$
|c_v-c_w|>2(r_{1}-r_i)+2\mu+(10+2n)\delta+4\epsilon
$$
then there is at most one vertex of color $i$ in each component of $\Gamma_{\mu}(Y)$. 
\end{corollary}
\begin{proof}
Assume $v$ and $w$ are in the same component of $\Gamma_{\mu}(Y)$. Then, by Proposition \ref{pathbound}, they are connected by an embedded path $P$ of length at most $2^{n}-2$. Therefore, Lemma \ref{centerdistance}.1 implies that $|c_v-c_w|\leq 2(r_{1}-r_i)+2\mu+(10+2n)\delta+4\epsilon$, contradicting the hypothesis. So, $v$ and $w$ must be in different components. 
\end{proof}

\subsection{The graph $\Gamma_{\mu}(\Xi)$} We will apply Proposition \ref{pathbound} and Corollary \ref{uniqueness} in the following situation. 
Let $\Xi=(\xi_v)_{v\in V}$ in $\mathbb K[G]$ be a finite family of group ring elements, such that no two are $\mathbb K$-scalar multiples of each other. 
Recall that the (geometric) support of $\xi_v=\sum_{g\in G}\xi_v^g\cdot g$ is the set $\mathcal X_v=\{g\cdot o\mid \xi_v^g\not=0\}$. 
Let $X=(\mathcal X_v)_{v\in V}$ be the family of these supports, and set $\Gamma_{\mu}(\Xi):=\Gamma_{\mu}(X)$. In this situation, there is a canonical partition of the index set $V$ into colors $V=V_1\sqcup\dots\sqcup V_n$ defined as follows. We declare that $v$ and $w$ have the same color if and only if there is $\lambda\in\mathbb K^*$ and $g\in G$ such that $\xi_v=\lambda g\xi_w$. Note that if $v$ and $w$ have the same color then $g\mathcal X_v=\mathcal X_w$ and hence $r_v=r_w$. Let $r_i$ be the radius of elements with color $i$.

\begin{corollary}[Uniform diameter bound+Uniqueness]
\label{diamcor}
Suppose that $G$ acts on a $\delta$-hyperbolic space $\mathcal H$ with displacement greater than $2\mu+(10+2n)\delta$. Let $\Xi=(\xi_v)_{v\in V}$ in $\mathbb K[G]$ be a family of group ring elements consisting of $n$ colors, and such that no two are $\mathbb K$-scalar multiples of each other. Then, 
\begin{itemize}
\item[1.] Each component of $\Gamma_{\mu}(\Xi)$ has diameter at most $2^n-2$.
\item[2.] If the minimal displacement is greater than $2(r_{1}-r_i)+2\mu+(10+2n)\delta$, then there is at most one vertex of color $i$ in each component of $\Gamma_{\mu}(\Xi)$.
\end{itemize}
\end{corollary} 
\begin{proof}
For part 1, pick $\epsilon$ small enough so that the minimal displacement is greater than $2\mu+(10+2n)\delta+4\epsilon$.
For two distinct elements $v,w\in V_i$, since $\xi_v$ is not a $\mathbb K$-scalar multiple of $\xi_w$ there is a non-trivial $g\in G$ such that $\lambda g\xi_v=\xi_w$. So, we have $r_v=r_w$. Since the displacement is greater than $3\delta$ we can---by Corollary \ref{equivcenter}---pick $\epsilon$-centers equivariantly, so that $gc_v=c_w$. Then $|c_v-c_w|=|c_v-gc_v|>2\mu+(10+2n)\delta+4\epsilon$. So, Proposition \ref{pathbound} applies. 

For part 2, pick $\epsilon$ small enough so that the minimal displacement is greater than $2(r_{1}-r_i)+2\mu+(10+2n)\delta+4\epsilon$, argue as in part 1, and apply Corollary \ref{uniqueness}.

\end{proof}

\subsubsection*{Components of $\Gamma_{\mu}(\Xi)$ and weak relations}
A finite collection of group ring elements $(\xi_v)_{v\in V}$ defines a {\it $\mu$-relation} if 
\begin{equation*}
\label{murelation}
\left|\sum_{v\in V}\xi_v\right|<\max_{v\in V}|\xi_v|-\mu.
\end{equation*}
The components of $\Gamma_{\mu}(\Xi)$ help us keep track of such $\mu$-relations.

\begin{lemma}
\label{relationcomponents}
If the family $\Xi=(\xi_v)_{v\in V}$ defines a $\mu$-relation, then every connected component $C$ of the graph $\Gamma_{\mu}(\Xi)$ that contains a vertex of $\Gamma_0(\Xi)$ defines its own $\mu$-relation $|\sum_{v\in C}\xi_v|<\max_{v\in C}|\xi_v|-\mu$. 
\end{lemma}
\begin{proof}
Set $d=\max_{v\in V}|\xi_v|$. Denote the complement of $C$ by $D=\Gamma_{\mu}-C$. Since $(\xi_v)_{v\in V}$ defines a $\mu$-relation, we have 
$$
\left|\sum _{v\in C}\xi_v+\sum_{v\in D}\xi_v\right|<d-\mu.
$$ 
Since $C$ is a component of $\Gamma_{\mu}(\Xi)$, the supports of $\sum_{v\in C}\xi_v$ and $\sum_{v\in D}\xi_v$ have no $\mu$-extremal points (that is, points $p$ with $|p|\geq d-\mu$) in common. Therefore, we must have
$$
\left|\sum _{v\in C}\xi_v\right|<d-\mu.
$$
Our choice of $C$ implies that $d=\max_{v\in C}|\xi_v|$, proving the lemma.  
\end{proof}

\section{Proof of Theorem \ref{intromaintheorem}\label{mainproof}}
We inductively define $\delta_n$ by $\delta_0=0$, $\delta_k=\delta_{k-1}+(2(4+\log_2(k+1))+7)\delta$. Then, \begin{eqnarray*}
\delta_n
&=&(15n+2\log_2((n+1)!))\delta.
\end{eqnarray*}
\begin{remark}
We will argue by induction on $n$. 
For most of the proof, all that will be relevant is that $(\delta_n)$ is an increasing sequence satisfying $\delta_n\geq(n+1)\delta$. For the very last inequality at the end of section \ref{mainproof}, we need to assume that $\delta_n-\delta_{n-1}\geq (2(4+\log_2(n+1))+7)\delta$. The above choice of $\delta_n$ is the smallest for which this holds.  
\end{remark}
\subsection{Restatement of the main theorem\label{restate}}


We will now restate Theorem \ref{intromaintheorem} in terms of colors and $\mu$-relations. We assume that $\xi_1,\dots,\xi_n$ are elements in the group algebra $\mathbb{K} [G]$ so
that there exist $\alpha_1, \ldots, \alpha_n$ in $\mathbb{K}
[G]$, not all zero, satisfying $| \sum \alpha_i \xi_i | < \max_i | \alpha_i \xi_i | -
\delta_n$. 

Recall that two elements $\xi_1$ and $\xi_2$ of $\mathbb{K} [G]$ have the
same color if there exists a trivial unit $\lambda g$ in
$\mathbb{K} [G]$ such that $\xi_1 - \lambda g \xi_2 = 0$. Certainly if two
of the $\xi_i$ have the same color, the conclusion of the theorem follows. So, from now on, we assume that the $\xi_i$ have different colors. Then, the index set $\{ 1, \ldots, n \}$ can be used as the set of colors. For each
$\alpha_i = \sum_{g \in G} \alpha^g_i \cdot g$, we let $A_i = \{g \in G \mid
\alpha^g_i \neq 0\}$ be its (algebraic) support in the group $G$. The group
ring elements $(\alpha^g_i g \xi_i)_{g \in A_i, 1 \leq i \leq n}$ are
distinct, no two of them are scalar multiples of each other, and for $i \neq
j$ they do not have the same color. We rename them $(\xi_v)_{v \in V}$ which
is a collection of elements in $\mathbb{K} [G]$ having $n$ colors. With this notation, as the $\xi_v$ are exactly the 
translates $\alpha_i^g\xi_i$, the sum $\sum_i\left(\sum_{g\in A_i}\alpha_i^g\right)\xi_i$ is now written as $\sum_{v\in V}\xi_v$ and the hypothesis $| \sum \alpha_i \xi_i | < \max_i | \alpha_i \xi_i |
- \delta_n$ implies
$$
\left| \sum_{v \in V} \xi_v \right| < \max_{v \in V} \left| \xi_v \right| - \delta_n.
$$
Recall that in this situation we say that the family $(\xi_v)_{v \in V}$ defines a
$\delta_n$-relation. Under this hypothesis, we will show that there exists a
vertex $v_*$ in $V$ and a subset $S \subset V$ made of elements of different
colors from the color of $v_*$ such that $\mathrm{diam} (\xi_{v_*} + \sum_{v
\in S} \xi_v) < \mathrm{diam} (\xi_{v_*}) - \delta$.


Summing up, Theorem \ref{intromaintheorem} will follow from

\begin{theorem} \label{maintheorem}
Suppose $G$ acts on a $\delta$-hyperbolic space $\mathcal H$ with minimal displacement $\rho>4\delta_n+(10+2n)\delta$. Let $\Xi=(\xi_v)_{v\in V}$ be a finite collection of elements in $\mathbb K[G]$ consisting of $n$ colors, and with no two elements being scalar multiples of each other. If $(\xi_v)_{v\in V}$ satisfies a $\delta_n$-relation,
then there is a subset $\{v_*\}\sqcup S$ of $V$ such that no element of $S$ has the same color as $v_*$ and
$$
\mathrm{diam}\left(\xi_{v_*}+\sum_{v\in S}\xi_v\right)< \mathrm{diam}(\xi_{v_*})-\delta. 
$$
\end{theorem}
\begin{remark}
The set $S\sqcup\{v_*\}$ will be constructed as the set of vertices of a component of $\Gamma_{\delta_k}(\Xi)$ for some $1\leq k\leq n$. 
\end{remark}
\begin{remark}
Note that the minimal displacement condition is 
$$
\rho/\delta>62n+10+8\log_2((n+1)!).
$$
So, $\rho/\delta>100\cdot\log_2 ((n+1)!)$ is enough to satisfy it. 
\end{remark}
\begin{remark}
We may assume that $\mathcal H$ has the additional property that every finite set has a $0$-center, which is the case if $\mathcal H$ is proper (i.e. closed balls are compact) or complete $\mathrm{CAT}(0)$. In other cases, one can choose a non-principal ultrafilter $\omega$ and replace $\mathcal H$ with the restricted ultraproduct $\mathcal H^{\omega}$, if necessary (see {\it Radii and centers} on page 6).
\end{remark}

\subsection{The case of trees\label{hogangeloniproof}}

We keep the notations of the previous paragraph unchanged, but assume that
the space $\mathcal{H}$ is a tree. This is the case $(\delta = 0)$ treated by
Hog-Angeloni in {\cite{hogangeloni}}. We slightly rephrase her proof.

Recall that $\Gamma_0 = \Gamma_0 (\Xi)$ is the graph whose vertices are the
elements $v$ in $V$ such that if $\mathcal X_v$ is the support of $\xi_v$, then $\mathcal X_v$
contains an extremal point, i.e. $| \xi_v | = d = \max_{v \in V} | \xi_v |$,
and there is an edge between $v$ and $w$ for each extremal point in $\mathcal X_v\cap\mathcal X_w$. 
The colors $\{ 1, \ldots, n \}$ are organized so that the $r_i$
(the radii of the $\mathcal X_i$) are in decreasing order: $r_1 \geq \cdots
\geq r_n$.

Let $\hat{\Gamma}_0$ be a component of $\Gamma_0$ containing a vertex $v_*$
with color $1$, the color of largest radius. By Lemma \ref{relationcomponents}, this graph defines a $0$-relation
\[ \left| \sum_{v \in \hat{\Gamma}_0} \xi_v \right| < \max_{v \in
\hat{\Gamma}_0} | \xi_v | . \]
By Corollary \ref{diamcor}, $v_*$ is the unique vertex in $\hat{\Gamma}_0$ of color $1$. So, removing $v_*$ from $\hat\Gamma_0$ and setting $S =
\hat{\Gamma}_0 - \{v_*\}$, we see that to conclude the proof of the theorem we need to estimate the diameter of the group ring element $\sum_{v\in\hat\Gamma_0}\xi_v=\xi_{v_*}+\sum_{v\in S}\xi_v$. 

Note that for any two adjacent vertices $v$ and $w$ in $\hat{\Gamma}_0$,
both centers $c_v$ and $c_w$ lie on a geodesic $[o, p]$ from the origin to an
extremal point. The oriented geodesics $[o, c_v]$ and $[o, c_w]$ coincide
along a segment of length $\min \{|c_v |, |c_w |\} \geq d - r_1 = | c_{v_*} |
.$ By following a path from $v_*$ to $v$ in the graph $\hat{\Gamma}_0$ and noting that
$r_1 \geq r_w$ for every vertex $w$ along that path, we conclude that
$c_{v_*}$ belongs to every geodesic $[o, c_v]$, and
\[ |c_{v_*} - c_v | = r_1 - r_v . \]
For every $v,$ the support of $\xi_v$ is contained in the closed ball $B
(c_{v_*}, r_1)$, so the same is true for the support of $\sum_{v \in\hat\Gamma_0} \xi_v$. Moreover, because of the $0$-relation, this support has no $0$-extremal points, i.e. it is contained in the ball $B (o, d-\alpha)$ for some $\alpha > 0$. So, applying Lemma \ref{twoballs}, we see that the the diameter of the support of $\sum_{v\in\hat\Gamma_0}\xi_v$ is at most
$$
d - \alpha + r_1 - |o - c_{v_*} | = 2 r_1 -\alpha= \mathrm{diam} (\xi_{v_*}) -
\alpha .
$$
This finishes the proof. 
\begin{remark}
In this result, we do not use that $\mathcal{H}$ is a combinatorial tree. It
might be an $\mathbb{R}$-tree with a free action of a surface group, for instance. But, for proving that the algorithm ends in a finite number of steps, we need that the diameter decreases by a given value, say $1$.
\end{remark}

\subsection{Proof of Theorem \ref{maintheorem} for hyperbolic spaces\label{proofsubsection}}



In order to prove Theorem \ref{maintheorem}, we keep the notations of \ref{restate} unchanged and proceed by induction on the number $n$ of colors. 

Let us fix some notation. We are given a collection of elements $\Xi=(\xi_v)_{v\in V}$ defining a $\delta_n$-relation. Denote $d=\max_{v\in V}|\xi_v|$. For any $\mu\geq 0$ we abbreviate $\Gamma_{\mu}=\Gamma_{\mu}(\Xi)$, so that vertices of $\Gamma_{\mu}$ are those $v\in V$ for which $|\xi_v|\geq d-\mu$. 
\subsection*{Base case}
If there is one color, then $\delta_1=17\delta$ and the minimum displacement is greater than $4\delta_1+12\delta$. So, Corollary \ref{diamcor} applies and implies that connected components of $\Gamma_{\delta_1}$ are points, and there are no $\delta_1$-relations. 
\begin{remark}
This case implies, in particular, that $\mathbb K[G]$ has no zero divisors. (The zero-divisor relation $0=\alpha\xi=\sum_{g\in G}\alpha^gg\xi$ is a $\delta_1$-relation defined by a finite collection of elements of a single color.)
\end{remark}
\subsection*{Inductive step}
Suppose now that we know the theorem for $n-1$ colors and want to prove it for $n$. By the inductive hypothesis, we may assume that $\Xi$  satisfies the following ``minimality'' condition:
\begin{itemize}
\item
No subset $W\subset V$ consisting of fewer than $n$ colors defines a $\delta_{n-1}$-relation. 
\end{itemize}
Since we will consider different values of the parameter $\mu$, it is useful to keep in mind that for $\mu<\mu'$ the graph $\Gamma_{\mu}$ is a subgraph of $\Gamma_{\mu'}$. In particular, for a vertex $v_0\in\Gamma_0$ (corresponding to an element $\xi_{v_0}$ whose support contains an extremal point) we have inclusions 
$$
v_0\in\Gamma_0\subset\Gamma_{\delta_{n-1}}\subset\Gamma_{\delta_n}. 
$$  
Denote by $\hat\Gamma_{\delta_{n-1}}$ and $\hat\Gamma_{\delta_n}$ the connected components of $\Gamma_{\delta_{n-1}}$ and $\Gamma_{\delta_n}$ containing the vertex $v_0$. Since $\Xi$ defines a $\delta_n$-relation (and $\delta_n>\delta_{n-1}$) Lemma \ref{relationcomponents} implies that $\hat{\Gamma}_{\delta_{n-1}}$ defines a $\delta_{n-1}$-relation, so minimality implies $\hat\Gamma_{\delta_{n-1}}$ contains all $n$ colors. 

\subsubsection*{Large and small colors}
Order the colors in decreasing order, $r_1\geq\dots\geq r_n$, and let $k$ be the largest integer for which $r_k\geq r_1-\delta_n$. We call $\{1,\dots,k\}$ the large colors and the rest small colors. 
The minimal displacement $> 4\delta_n+(10+2n)\delta$ assumption implies, by Corollary \ref{diamcor}.2, that each large color appears exactly once in the connected graph $\hat\Gamma_{\delta_n}$.
Thus, we can identify the set of large colors $\{1,\dots,k\}$ with a set of $k$ vertices $\{v_1,\dots,v_k\}$ in $\hat\Gamma_{\delta_n}$. 
Since each color appears in $\hat\Gamma_{\delta_{n-1}}$, we conclude
$$
\{v_1,\dots,v_k\}\subset\hat\Gamma_{\delta_{n-1}}\subset\hat\Gamma_{\delta_n}.
$$

\subsubsection*{Centers}
Next, we compare the positions of the centers of the sets $\mathcal X_v$. We will prove
\begin{proposition}
\label{centergromovproduct}
Suppose $v$ and $w$ are vertices in $\hat \Gamma_{\delta_n}$. Then 
$$
\left<c_v,c_w\right>\geq d-{\delta_n+\delta_{n-1}\over 2}-r_1-(4+\log_2(n+1))\delta.
$$
\end{proposition}
We need the following lemmas. 
\begin{lemma}
\label{largeandsmall}
Suppose that $v\in\hat\Gamma_{\delta_n}$ and $p\in\mathcal X_v$ is a $\delta_n$-extremal point of $\cup_{v\in V}\mathcal X_v$. 
\begin{enumerate}
\item[1.] 
If the color of $v$ is small, then $\left<p,c_v\right>\geq d-r_1$.
\item[2.] 
If the color of $v$ is large, then $\left<p,c_v\right>\geq d-{\delta_n+\delta_{n-1}\over 2}-r_1$.
\end{enumerate}
\end{lemma}
\begin{proof}
Since $p$ is $\delta_n$-extremal, Lemma \ref{fellowtravel1} implies
$$
\left<p,c_v\right>\geq{|\mathcal X_v|+|p|\over 2}-r_v\geq{|\mathcal X_v|+d-\delta_n\over 2}-r_v.
$$
If the color of $v$ is small, then $r_v\leq r_1-\delta_n$ and $|\mathcal X_v|\geq d-\delta_n$ so that $\left<p,c_v\right>\geq d-r_1$. If the color of $v$ is large, then $v\in\hat\Gamma_{\delta_{n-1}}$ so $|\mathcal X_v|\geq d-\delta_{n-1}$ and $-r_v\geq-r_1$, implying that $\left<p,c_v\right>\geq{d-\delta_{n-1}+d-\delta_n\over 2}-r_1$.
\end{proof}
\begin{lemma}
\label{adjacentvertices}
If $v$ and $w$ are adjacent vertices in $\hat\Gamma_{\delta_n}$, then
$$
\left<c_v,c_w\right>\geq d-{\delta_n+\delta_{n-1}\over 2}-r_1-\delta.
$$
\end{lemma}
\begin{proof}
If $v$ and $w$ are adjacent, then there is a $\delta_n$-extremal point $p\in\mathcal X_v\cap\mathcal X_w$. By definition of $\delta$-hyperbolicity, $\left<c_v,c_w\right>\geq \min(\left<c_v,p\right>,\left<p,c_w\right>)-\delta$ and the right hand side is $\geq d-{\delta_n-\delta_{n-1}\over 2}-r_1-\delta$ by Lemma \ref{largeandsmall}. 
\end{proof}
\begin{lemma}
\label{smallpath}
Suppose $v_0,v_1,\dots,v_m$ is a path in $\hat\Gamma_{\delta_n}$ consisting of vertices with small colors. Then 
$$
\left<c_{v_0},c_{v_m}\right>\geq d-r_1-(n+1)\delta.
$$
\end{lemma}
\begin{proof}
We may assume the path is embedded. Then, by Proposition \ref{pathbound}, $m<2^n$. Since the path is in $\hat\Gamma_{\delta_n}$, there is a sequence of $\delta_n$-extremal points $p_i$ such that $p_i\in\mathcal X_{v_i}\cap \mathcal X_{v_{i+1}}$. Applying the Gromov product inequality to the sequence of ($<2^{n+1}$) points $v_0,p_0,v_1,p_1,\dots,p_{i-1},v_m$ and using Lemma \ref{largeandsmall}.1 gives the result. 
\end{proof}
\begin{remark}
Note that $\delta_n/\delta$ is superlinear in $n$, and in particular $(\delta_n+\delta_{n-1})/2>n\delta$ for $n>0$. Therefore, the right hand side in Lemma \ref{smallpath} is larger than in Lemma \ref{adjacentvertices}. 
\end{remark}
\begin{proof}[Proof of Proposition \ref{centergromovproduct}]
Let $v=v_0,\dots,v_m=w$ be an embedded path of vertices from $v$ to $w$. Note that it contains at most $n$ large vertices. For every maximal subpath that consists entirely of small vertices (there are at most $n+1$ such subpaths), throw out all but the initial and final vertex. We are left with a sequence of at most $2(n+1)+n$ vertices $v=v_0',v'_1,\dots,v'_k=w$, where any two consecutive vertices are either adjacent in $\hat\Gamma_{\delta_{n}}$, or the endpoints of a path consisting of small vertices. In either case $\left<c_{v_i'},c_{v'_{i+1}}\right>\geq d-{\delta_n+\delta_{n-1}\over 2}-r_1-\delta$ by the above two lemmas. Therefore, the Gromov product inequality implies $\left<c_v,c_w\right>\geq d-r_1-{\delta_n+\delta_{n-1}\over 2}-(1+\lceil\log_2(3n+1)\rceil)\delta$. Since $1+\lceil\log_2(3n+1)\rceil\leq 2+\log_2(3n+1)\leq 2+\log_23+\log_2(n+1)\leq 4+\log_2(n+1)$, we get the result.
\end{proof}

The correction term $(4+\log_2(n+1))\delta$ in Proposition \ref{centergromovproduct} recurs throughout the rest of the proof so, to save space and improve readability, set $L(n)=4+\log_2(n+1)$. 

Note that for any $w\in\hat\Gamma_{\delta_n}$ we have\footnote{Either $w$ is small, in which case $|c_w|\geq d-\delta_{n}-r_w\geq d-r_1$ or $w\in\hat\Gamma_{\delta_{n-1}}$ and then $|c_w|\geq d-\delta_{n-1}-r_w\geq d-\delta_{n-1}-r_1$.} 
\begin{equation}
\label{centersize}
|c_w|\geq d-\delta_{n-1}-r_1.
\end{equation} 
So, since $\delta_n\geq\delta_{n-1}$, we can pick $c'_w\in [o,c_w]$ such that $|c'_w|=d-{\delta_n+\delta_{n-1}\over 2}-r_1-L(n)\delta$. 
\begin{lemma}
For any vertices $v,w$ in $\hat\Gamma_{\delta_n}$ we have 
\begin{eqnarray*}
|c'_v-c'_w|&\leq&\delta,\\
|c_w-c_w'|&\geq&{\delta_n-\delta_{n-1}\over 2}+L(n)\delta.
\end{eqnarray*}
\end{lemma}
\begin{proof}
The first inequality follows from Proposition \ref{centergromovproduct} and hyperbolicity. 
The second inequality is obtained by subtracting $|c_w'|$ from both sides of (\ref{centersize}). 
\end{proof}

\subsubsection*{Bounding the diameter of support}
Recall that the family $\Xi$ defines a $\delta_n$-relation, so the support of $\sum_{v\in\hat\Gamma_{\delta_n}}\xi_v$ is contained in the ball $B(o,d-\delta_n)$. Let $c^*$ be the point in $[o,c_{v_1}]$ such that $|c^*|=d-\delta_{n-1}-r_1-L(n)\delta$. Then---by definition---we have $|c^*-c'_{v_1}|={\delta_n-\delta_{n-1}\over 2}$. 
Next, we will show that 
the ball $B(c^*,r_1+(L(n)+2)\delta)$ also contains the support of $\sum_{v\in\hat\Gamma_{\delta_n}}\xi_v$: 
\begin{lemma} 
\label{radiuslemma}
For $w\in\hat\Gamma_{\delta_n}$ and a point $p\in\mathcal X_w$ that is not $\delta_n$-extremal we have 
$$
|c^*-p|\leq r_1+(L(n)+2)\delta.
$$
\end{lemma}


\begin{proof}
Let $p'$ be the projection of $p$ onto $[o,c_w]$. There are two cases. 
\begin{itemize}
\item
If $p'\in[o,c_w']$ then 
$$
{\delta_n-\delta_{n-1}\over 2}+|c_w'-p|\leq |c_w-c'_w|+|c_w'-p|\stackrel{h}\leq |c_w-p|+\delta\leq r_1+\delta
$$

\item
If $p'\in[c_w',c_w]$ then, since $p$ is not $\delta_n$-extremal, 
$$
\left(d-r_1-{\delta_n+\delta_{n-1}\over 2}-L(n)\delta\right)+|c_w'-p|=|c'_w|+|c_w'-p|\stackrel{h}\leq|p|+\delta\leq d-\delta_n+\delta.
$$
\end{itemize}
\begin{remark}
The inequalities following from hyperbolicity are denoted $\stackrel{h}\leq$ for emphasis. 
\end{remark}
Rearranging, we see that in either case we have obtained the inequality 
$$
|c_w'-p|\leq r_1
-{\delta_{n}-\delta_{n-1}\over 2}
+(L(n)+1)\delta.
$$ 
Therefore, 
\begin{eqnarray*}
|c^*-p|&\leq&|c^*-c'_{v_1}|+|c'_{v_1}-c'_{w}|+|c'_w-p|,\\
&\leq&\left({\delta_n-\delta_{n-1}\over 2}\right)+\delta+\left(r_1-{\delta_n-\delta_{n-1}\over 2}+(L(n)+1)\delta\right),\\
&=&r_1+(L(n)+2)\delta. 
\end{eqnarray*}
which finishes the proof. 
\end{proof}

 We've shown that the support of $\sum_{v\in\hat\Gamma_{\delta_n}}\xi_v$ is contained in the intersection of balls $B(o,d-\delta_n)\cap B(c^*,r_1+(L(n)+2)\delta)$. Therefore, Lemma \ref{twoballs} implies the diameter of the support is 
\begin{eqnarray*}
&\leq&(d-\delta_n)+(r_1+(L(n)+2)\delta)-(d-r_1-\delta_{n-1}-L(n)\delta)+2\delta\\
&=&(2r_1-2\delta)-\delta_n+\delta_{n-1}+(2L(n)+6)\delta\\
&\leq&\mathrm{diam}(\xi_{v_1})-\delta,
\end{eqnarray*}
where for the last inequality we've used the fact that $2r_1\leq\mathrm{diam}(\xi_{v_1})$ by Lemma \ref{diamvrad} and the inductive description of $\delta_n$. Since $v_1$ is the unique vertex of color $1$ in $\hat\Gamma_{\delta_n}$, this establishes the theorem.

\section{Applications}
We now apply the algorithm. 
The main hypothesis in this section is:

\vspace{0.4cm}

\noindent ($\mathcal{H}_{n}$):\hspace{0.1cm} {\it $G$ acts on a $1$-hyperbolic space with displacement more than $100\log_2((n+1)!)$.}

\vspace{0.4cm}

\noindent
Let $\mathrm{E}_n (\mathbb{K} [G])$ be the subgroup of elementary matrices in $\mathrm{GL}_n
(\mathbb{K} [G])$. We begin with a linear algebraic lemma. 


\begin{lemma}
\label{elementary}
Suppose the group $G$ satisfies $\mathcal H_{n}$. 
Let $\xi=(\xi_1,\dots,\xi_n)\in\mathbb K[G]^n$.
\begin{enumerate}
\item[0.]
If the coordinates of $\xi$ are linearly dependent over $\mathbb K[G]$, then the $\mathrm{E}_n(\mathbb K[G])$-orbit of $\xi$ contains a vector which has at least one coordinate equal to $0$. 
\item[1.]
If there is $\alpha=(\alpha_1,\dots,\alpha_n)\in\mathbb K[G]^n$ satisfying $\alpha\cdot\xi=\sum\alpha_i\xi_i=1$, then the $\mathrm{E}_n(\mathbb K[G])$-orbit of $\xi$ contains $(\lambda g,0,\dots,0)$ for some $\lambda\in\mathbb K^*$ and $g\in G$. 
\end{enumerate}
\end{lemma}

\begin{remark}
Elements satisfying 1 are called unimodular vectors by Bass and play a crucial role in algebraic K-theory. 
\end{remark}

%
%
%
%
%
%

\begin{proof}
0. Pick a vector $\xi'$ in the $\mathrm{E}_n(\mathbb K[G])$-orbit of $\xi$ that minimizes the sum of diameters of its coordinates. The coordinates of $\xi'$ are still linearly dependent. If none of them are zero, then Theorem \ref{intromaintheorem} would let us reduce the sum of diameters, contradicting the minimality assumption.

1. We argue the same way. Pick a vector $\xi'=\xi U$ in the $\mathrm{E}_n(\mathbb K[G])$-orbit of $\xi$ that minimizes the sum of diameters of {\it nonzero} coordinates. Then 
$$
1=\alpha\cdot\xi=\alpha U^{-t}\cdot \xi U=\alpha'\cdot\xi'.
$$ 
Pick $i$ with $\alpha_i'\xi_i'\not=0$. If $|\alpha_i'\xi_i'|>0$ then we can apply Theorem \ref{intromaintheorem} to reduce the sum of diameters of nonzero coordinates of $\xi'$, contradicting minimality. So, $|\alpha'_i\xi'_i|=0$, i.e. $\xi_i'$ is a unit. By \cite{delzant}, $\mathbb K[G]$ only has trivial units, so $\xi_i'=\lambda g$  for some $\lambda\in\mathbb K^*$ and $g\in G$. The conclusion follows by applying elementary transformations to $\xi'$. 
\end{proof}
  
\subsection{Freeness}
Lemma \ref{elementary} enables the study of finitely generated submodules of free modules.

\begin{theorem}
\label{freemodules}
Assume the group $G$ satisfies ${\mathcal{H}_{n}}$.
\begin{enumerate}
\item[1.] 
Every $n$-generated ideal in $\mathbb{K} [G]$ is a free $\mathbb{K}[G]$-module.
\item[2.] 
Every $n$-generated submodule of a free $\mathbb{K} [G]$-module is a free $\mathbb{K} [G]$-module.
\end{enumerate}
\end{theorem}

\begin{proof}
1. Suppose we have shown that ideals generated by fewer than $n$ elements are free, and let $\mathcal{I}$ be an ideal in $\mathbb{K} [G]$ generated $n$ elements $\xi_1,\dots,\xi_n$. Consider the map
\begin{eqnarray*}
\mathbb K[G]^n&{\rightarrow}&\mathbb{K}[G],\\
(\alpha_1, \ldots, \alpha_n) & \mapsto & \alpha_1 \xi_1+\ldots +\alpha_n\xi_n.
\end{eqnarray*}
If this map is injective, then it provides an isomorphism from the free module $\mathbb K[G]^n$ to the ideal $\mathcal I$. If it is not injective, then there is a non-trivial relation $\alpha_1\xi_1 + \ldots + \alpha_n \xi_n = 0$. In other words the family $\xi_1,\dots,\xi_n$ is linearly dependent. By Lemma \ref{elementary}.0 we can do elementary transformations to replace $\xi_1,\dots,\xi_n$ by a generating set for $\mathcal I$ consisting of $n-1$ elements. Therefore $\mathcal I$ is free. 

2. Suppose $M \subset \mathbb{K} [G]^m$ is an $n$-generated submodule. Note that $M \otimes_{\mathbb{K} [G]} \mathbb{K}$ is a finite dimensional $\mathbb{K}$-vector space of some dimension $d\leq n$. We argue by induction on $d$. If $M \neq 0$ there is a projection to a factor $p : \mathbb{K} [G]^m\rightarrow \mathbb{K} [G]$ such that $p (M)$ is non-trivial. Since $p (M)$ is an $n$-generated ideal, it is free as a $\mathbb{K} [G]$ module by part 1. It follows that the module $M$ maps onto a non-zero free $\mathbb{K} [G]$-module and, a fortiori, onto $\mathbb{K} [G]$. Therefore, the module $M$ splits as $M \cong M' \oplus \mathbb{K} [G]$. Note that $\dim_{\mathbb{K}} (M\otimes_{\mathbb{K} [G]} \mathbb{K}) - 1 = \dim_{\mathbb{K}} (M' \otimes_{\mathbb{K} [G]} \mathbb{K})$, and the result follows.
\end{proof}

From this theorem and Stallings result on groups with infinitely many ends, we can also deduce:

\begin{corollary}
\label{stallings}
Under the same hypotheses, every $n$ generated subgroup of the group $G$ is free.
\end{corollary}

\begin{proof}
Let $\mathbb{K}$ be any field, for instance $\mathbb{K}=\mathbb{F}_2$. Suppose $(g_1, \ldots, g_n) = H<G$ is an $n$-generated subgroup. Then its augmentation ideal $(g_1 - 1, \ldots, g_n - 1)$ is a free ideal in $\mathbb{K} [H]$ by Theorem \ref{freemodules}.1. Clearly the group $H$ is torsion-free, so by 3.14 of \cite{dicksdunwoody}, the group $H$ is a free group.
\end{proof}

\begin{remark}
For the convenience of the reader, let us recall the argument of Dicks and Dunwoody (\cite{dicksdunwoody}). The main ingredient behind the passage from ideals to subgroups is Stallings theorem on ends (\cite{stallings}). If the augmentation ideal $\mathcal{I}$ is free and finitely generated, we have a length one resolution $0\rightarrow \mathcal{I} \rightarrow \mathbb{F}_2 [G] \rightarrow\mathbb{F}_2 \rightarrow 0$ of $\mathbb F_2$ by finitely generated, free modules. Hence $H^*(G;\mathbb F_2[G])$ is non-zero in degree $0$ or $1$ (or both).\footnote{The reason is that if a chain complex $C$ of finitely generated free $\mathbb F_2[G]$-modules is not acyclic, the its $\mathbb F_2[G]$-hom dual $C^*$ is not acyclic either, since $C^{**}=C$.} Since $G$ is non-trivial and torsion-free, it is infinite, so---using the identification of $H^0$ with $G$-invariants--- we see that $H^0(G;\mathbb F_2[G])=(\mathbb F_2[G])^G=0$. Therefore $H^1 (G, \mathbb{F}_2 [G]) \neq 0$: the group $G$ has several ends. As the group $G$ is torsion free, Stallings theorem implies that either it is infinite cyclic or that it splits as a free product $G = G_1 \ast G_2$. By Grushko's theorem, the groups $G_1, G_2$ have smaller rank, and one can conclude by induction. This argument was used by Stallings to prove that a group of cohomological dimension one is a free group.
\end{remark}

\begin{remark}
The qualitative result---large displacement ($\geq \rho$) implies that every $n$-generated subgroup is free---is not new. It has been stated by Gromov \cite{gromovhyperbolic}, and proofs have been given by Arzhantseva \cite{arzhantseva} and Kapovich-Weidmann \cite{kapovichweidmann}. The best previous quantitative result---due to Gromov \cite{gromovexpanders} p.763---required  $\rho/\delta>10^6\cdot n\log_2n$. Conjecturally, $\rho/\delta>1000\cdot \log_2(n+100)$ should be enough (see 5.3 in \cite{gromovhyperbolic}).

\end{remark}

\subsection{The group $\text{GL}_n (\mathbb{K} [G])$}
Recall that $\text{GE}_n (\mathbb{K} [G])$ is the subgroup of $\text{GL}_n (\mathbb{K} [G])$ generated by elementary and diagonal matrices. 

\begin{theorem}
\label{gltheorem}
Assume $G$ satisfies ${\mathcal{H}_{n}}$. Then 
$$
\mathrm{GL}_n (\mathbb{K} [G])=\mathrm{GE}_n (\mathbb{K} [G]).
$$
\end{theorem}

\begin{proof}
We copy the usual proof of the fact that $\text{GL}_n (\mathbb{Z})$ is generated by elementary matrices and diagonal matrices with entries $\pm 1$. Let $X = (\xi_{ij})$ be in $\text{GL}_n (\mathbb{K} [G])$, and choose $A=(\alpha_{ij})$ in $\text{GL}_{_n} (\mathbb{K} [G])$ such that $AX = 1$. As $\sum_{i = 1}^n \alpha_{1 i} \xi_{i 1} = 1$, we can apply Lemma \ref{elementary}.1, and deduce that there is a matrix $U$ in $\text{GE}_n (\mathbb{K} [G])$ such that the first row of $XU$ is $(u,0,\dots,0)$ for some unit $u\in\mathbb K[G]$. 
Left multiplying $XU$ by a product of elementary matrices, say $V \in \text{E}_n (\mathbb{K} [G])$, we obtain a matrix $VXU$ of the form 
$$
\left(\begin{array}{cc}
    u & 0\\
    0 & Y
\end{array}\right),
$$ 
where $Y$ is a matrix in $\text{GL}_{n - 1} (\mathbb{K}[G])$. So, the theorem follows by induction on $n$.
\end{proof}

{
\renewcommand{\thetheorem}{\ref{fg}}

\begin{theorem}
Assume $\mathcal H_{n}$.  If $\mathbb K$ is finite and $G$ is finitely generated, then $\mathrm{GL}_n (\mathbb{K} [G])$ is finitely generated.
\end{theorem}
\addtocounter{theorem}{-1}
}

\begin{proof}
Recall that, because of the large displacement assumption (greater than $4\delta$ is enough), all units in $\mathbb K[G]$ are trivial by \cite{delzant}. Let $e_1,\dots,e_n$ be the standard basis for $\mathbb K[G]^n$. The group $\text{GE}_n(\mathbb K[G])$ is generated by the finitely many elementary transformations of the form $(e_i\mapsto e_i+e_j)$ and multiplication of basis elements by units of the form $\lambda g$, where $\lambda$ is in $\mathbb K^*$ and $g$ is in a generating set for $G$. So, the corollary follows from the previous theorem.  
\end{proof}
\begin{remark}
The proof only uses that $\mathbb K^*$ is finitely generated. However, no example of an infinite field with finitely generated $\mathbb K^*$ is known. 
\end{remark}

Theorems \ref{freemodules}.1 and \ref{gltheorem} together establish Theorem \ref{mainfirtheorem} from the introduction. 

\subsection{Submodules of free $\mathbb{Z} [G]$-modules\label{zsubsection}}
Until now we've been studying the group algebra $\mathbb{K} [G]$ with coefficients in a field $\mathbb{K}$. Our next goal is to extend freeness results to the integral group ring $\mathbb{Z} [G]$. We follow the method Bass used in \cite{bass} to show projective $\mathbb Z[F]$-modules are free,
but with two differences. First, the article \cite{bass} studies group rings with coefficients in a principal ideal domain. 
We restrict ourselves to the ring $\mathbb{Z}$. Second, the main hypothesis of \cite{bass} is that $M$ is a projective module. Instead, we use 


\begin{itemize}
\item[$(\star)$] The module $M$ embeds as a submodule of a free module $\oplus\mathbb{Z} [G]$ such that the quotient $(\oplus \mathbb{Z}[G])/M$ is torsion free as an abelian group.
\end{itemize}


This condition is probably known to specialists but hard to locate in the literature. 
It is a weakening of the more familiar ``projective''. Indeed, if $M$ is a projective module, there exists a module $N$ such that $M \oplus N = \oplus \mathbb{Z} [G]$ is a free module, in particular the quotient $N$ is torsion free. It is useful for obtaining topological consequences thanks to the following important example.

\begin{example}
Let $X$ be a finite, connected, non-empty cell complex with fundamental group $G$, $\widetilde{X}$ its universal cover and $C_{*} := C_{*} (\widetilde{X}; \mathbb{Z})$ its cellular chain complex. Then, the kernels of the boundary maps $\partial : C_k \rightarrow C_{k - 1}$ and of the augmentation map $C_0 \rightarrow \mathbb{Z}$ all satisfy condition $(\star)$, as their quotients are submodules of the free modules $C_{k - 1}$ and of $\mathbb{Z}$, respectively (see the proof of Theorem \ref{cdim}). In particular, the augmentation ideal of $G$, the relation module of a generating set for $G$, and the second homotopy module of a presentation $2$-complex for $G$ all satisfy $(\star)$. 
\end{example}
\begin{remark}
In \cite{bass}, Bass shows that projective modules over $\mathbb K[\mathbb Z\times F_m]\cong\mathbb K[t,t^{-1}][F_m]$ are free. As pointed out by the referee, in our situation it may also be possible to use coefficients in the PID $\mathbb K[t,t^{-1}]$ to show that $n$-generated projective modules over $\mathbb K[\mathbb Z\times G]\cong\mathbb K[t,t^{-1}][G]$ are free, but we do not pursue this here.  
\end{remark}
Here is the main theorem of this subsection. 
\begin{theorem}
\label{freezmodules}
Suppose the group $G$ satisfies $\mathcal{H}_{n}$. Let $M$ be a $n$-generated submodule of a free module $\oplus \mathbb{Z}G$ such that $(\oplus \mathbb{Z}G)/ M$ is torsion free as an abelian group. Then $M$ is free.
\end{theorem}

\begin{remark}
Some assumption such as $(\star)$ is necessary to establish freeness: the ideal in $\mathbb{Z} [t, t^{- 1}]$ generated by $u = 2$ and $v = t - 1$ is not free, as it satisfies the relation $(t - 1) u - 2 v = 0$ and is not generated by a single element. 
\end{remark}

To ellucidate the role of $(\star)$ in the proof, we recall some basics on abelian groups. 
\subsubsection*{On abelian groups and torsion}
For an abelian group $A$, let $A_{\mathbb Q}:=A\otimes_{\mathbb Z}\mathbb Q$ be its {\it rationalization} and $A_p:=A\otimes_{\mathbb Z}\mathbb F_p$ its {\it mod $p$ reduction}. An inclusion of abelian groups induces homomorphisms of rationalizations and mod $p$ reductions. The latter may no longer be an inclusion. In our proof, condition ($\star$) will be relevant because it implies injectivity of the induced map on mod $p$ reductions via the last part of the following lemma which summarizes basic properties of rationalizations and mod $p$ reductions that we will need. 
\begin{lemma}
\label{abgroup}
For any abelian group $A$, 
\begin{itemize}
\item[1.]
the mod $p$ reduction can be expressed as $A_p=A\otimes_{\mathbb Z}\mathbb F_p\cong A/pA$.
\item[2.]
If $A$ is torsion-free, then the natural map $A\rightarrow A_{\mathbb Q}$ is an embedding. 
\end{itemize} 
Let $A\hookrightarrow B$ be an embedding of abelian groups. Then
\begin{itemize}
\item[3.]
the induced map of rationalizations $A_{\mathbb Q}\rightarrow B_{\mathbb Q}$ is injective, and
\item[4.]
if $B/A$ has no $p$-torsion then the induced map $A_p\rightarrow B_p$ is injective. 
\end{itemize}
\end{lemma} 
\begin{proof}
The failure of tensor products to preserve injectivity is measured by Tor as appying $A\otimes_{\mathbb Z}-$ to an exact sequence of abelian groups $0\rightarrow B\rightarrow C\rightarrow D\rightarrow 0$ gives the exact sequence
$$
\mathrm{Tor}(A,D)\rightarrow A\otimes_{\mathbb Z}B\rightarrow A\otimes_{\mathbb Z}C\rightarrow A\otimes_{\mathbb Z}D\rightarrow 0.
$$ 
The basic properties of Tor we will use can be found in 3A.5 of \cite{hatcherbook}.
For the first point, apply $A\otimes_{\mathbb Z}-$ to the exact sequence $0\rightarrow\mathbb Z\stackrel p\rightarrow\mathbb Z\rightarrow\mathbb F_p\rightarrow 0$. For the second, apply it to the exact sequence $0\rightarrow\mathbb Z\rightarrow\mathbb Q\rightarrow\mathbb Q/\mathbb Z\rightarrow 0$ and note that the tor term $\mathrm{Tor}(A,\mathbb Q/\mathbb Z)$ vanishes because $A$ is torsion-free. For the third point, apply $-\otimes_{\mathbb Z}\mathbb Q$ to $0\rightarrow A\rightarrow B\rightarrow B/A\rightarrow 0$ and note that $\mathrm{Tor}(B/A,\mathbb Q)$ vanishes beecause $\mathbb Q$ is torsion-free. For the forth point, apply $-\otimes_{\mathbb Z}\mathbb F_p$ to the same sequence and note that the tor term $\mathrm{Tor}(B/A,\mathbb F_p)=\ker(B/A\stackrel{p}\rightarrow B/A)$ vanishes because $B/A$ has no $p$-torsion. 
\end{proof}

\subsubsection*{A `local-to-global' principle for rings} We can now prove Theorem \ref{freezmodules}. Note that $\mathbb Z[G]$ is torsion-free as an abelian group, so any submodule of $\oplus\mathbb Z[G]$ is, as well. Moreover, the rank of a finitely generated free $\mathbb K[G]$-module is the dimension of the vector space of coinvariants $\mathbb K[G]^m\otimes_{\mathbb K[G]}\mathbb K=\mathbb K^m$. So, Theorem \ref{freezmodules} is a consequence of Theorem \ref{freemodules}.2, Theorem \ref{gltheorem} and the $R=\mathbb Z[G]$ case of the following purely ring theoretic proposition.  

\begin{proposition}
\label{bassprop}
Suppose $R$ is a ring satisfying the following properties: 
\begin{itemize}
\item[$(\mathbb{Q})$] 
every $n$-generated $R_{\mathbb Q}$-submodule of $\oplus R_{\mathbb Q}$ is free of unique rank\footnote{A finitely generated $R$-module is free of unique rank if it is isomorphic to $R^m$ for a unique $m$.},
\end{itemize}
and for all primes $p$
\begin{itemize}
\item[$(p_0)$]
every $n$-generated $R_p$-submodule of $\oplus R_p$ is free of unique rank, and
\item[$(p_1)$] 
for $m\leqslant n-1$ we have $\mathrm{GL}_m(R_p)=\mathrm{GE}_m(R_p)$. 
\end{itemize}
If $M$ is an $n$-generated $R$-submodule of $\oplus R$ such that both $M$ and $(\oplus R)/M$ are torsion-free abelian groups, then $M$ is a free $R$-module. 
\end{proposition}

\begin{proof}
Since $M$ is torsion-free, the map $M\rightarrow M_{\mathbb Q}$ is an embedding. So, we can think of $M$ as a subgroup of $M_{\mathbb Q}$. 
Let $x_1, \dots, x_n$ generate the $R$-module $M$. By Lemma \ref{abgroup}.3, $M_{\mathbb{Q}}$ is an $R_{\mathbb Q}$-submodule of the free $R_{\mathbb Q}$-module $\oplus R_{\mathbb Q}$. So, by our ($\mathbb{Q}$) hypothesis, it is a free $R_{\mathbb Q}$-module. If this $R_{\mathbb Q}$-module is of rank $n$ (not $\leqslant n - 1$), then the family $x_1, \ldots, x_n$ is an $R_{\mathbb Q}$-basis for $M_{\mathbb Q}$. Therefore, it is an $R$-basis for $M$, and hence $M$ is free of rank $n$, and we are done. 

Otherwise, $M_{\mathbb{Q}}$ is of rank $m < n$. Let $y_1,\ldots, y_m$ be an $R_{\mathbb Q}$-basis for $M_{\mathbb Q}$. Clearing denominators, we may assume that each $y_i$ is in $M$. Then, the $y_i$ generate a free, rank $m$ $R$-submodule $Y:=\left<y_1,\dots,y_m\right>$ of $M$. Next, since the $y_i$ form an $R_{\mathbb Q}$-basis for $M$, each $x_i$ can expressed as a $R_{\mathbb Q}$-linear combination of the $y_i$. We can clear denominators in these expressions and find a single positive integer $k$ such that---for all $j$---$kx_j$ is a $R$-linear combination of the $y_i$. In summary we have obtained a free $R$-module $Y$, a positive number $k$, and inclusions
$$
kM\subset Y\subset M.
$$

Let $k\geq 1$ be the smallest number such that there is a free $R$-module $Y$ of rank $m$ with $kM\subset Y\subset M$. If $k=1$, then we are done since then $M=Y$ is a free $R$-module. So, towards a contradiction, suppose $k>1$. We will find a free, rank $m$ $R$-module $Y'$ and $1\leq k'<k$ such that $k'M\subset Y'\subset M$.  

To that end, pick a prime $p$ dividing $k$ and let $f:Y_p\rightarrow M_p$ be the mod $p$ reduction of the second inclusion above. By Lemma \ref{abgroup}.4 and the hypothesis that $(\oplus R)/M$ is torsion-free as an abelian group, the $R_p$-module $M_p$ embeds in the free $R_p$-module $\oplus R_p$. Since $M_p$ is generated by $n$-elements, our hypothesis ($p_0$) shows that it is free. The $R_p$-module $f (Y_p)$ is an $R_p$-submodule of $M_p$, and is generated by $m<n$ elements, so it is again free by $(p_0)$. Thus, we have a splitting
$$
Y_p \cong\ker(f) \oplus\mathrm{im} (f) .
$$
Note that the left hand side is also a free $R_p$-module, since it the mod $p$ reduction of the free $R$-module $Y$. The splitting shows that $\ker (f)$ is an $m$-generated $R_p$-submodule of a free $R_p$-module, so it is also free. Therefore, we can pick an $R_p$-basis $z_1,\dots,z_k,z_{k+1},\dots,z_m$ for the $R_p$-module $Y_p$ so that the first $k$ elements are a basis for the kernel of $f$. 
By our ($p_1$) hypothesis, there exists a matrix $U$ in $\mathrm{E}_m  (R_p)$ transforming the mod $p$ reduction of the family $\{y_i\}$ to the family $\{u_i z_i\}$ where the $u_i$ are units in $R_p$. As a matrix in $\mathrm{E}_m (R_p)$ is a product of elementary matrices, we can lift $U$ to $\mathrm{E}_m  (R)$ transforming the family $\{y_i\}$ to a family $\{y'_i\}$ such that the reduction mod $p$ of $y'_i$ is $u_i z_i$. Let $P:=\left<y_1',\dots,y_k'\right>$ and $Q:=\left<y_{k+1}',\dots,y_m'\right>$, so that 
$$
Y=P\oplus Q,
$$
and on mod $p$ reductions
\begin{itemize}
\item
$P_p\rightarrow M_p$ is the zero map, i.e. $P\subset pM$, while
\item
$Q_p\rightarrow M_p$ is injective, i.e. $pQ=Q\cap pM$.
\end{itemize}
For $x\in M$, the inclusion $kM\subset Y=P\oplus Q$ lets us write $kx=a+b$ where $a\in P$ and $b\in Q$. Since $P\subset pM$, we have\footnote{For an element $a$ in a torsion-free abelian group, if there is an element $a_0$ such that $a=pa_0$, then there is a unique such element. We call this element ${a\over p}$.} ${a\over p}\in {1\over p}P\subset M$ and hence $b=p({k\over p}x-{a\over p})\in Q\cap pM=pQ$. Therefore, ${b\over p}\in Q$. Since ${k\over p}x={a\over p}+{b\over p}$ we have obtained the inclusions
$$
{k\over p}M\subset {1\over p}P\oplus Q\subset M.
$$
Since ${1\over p}P\oplus Q$ is a free $R$-module of rank $m$ (with basis $\{{y_1'\over p},\dots,{y_k'\over p},y'_{k+1},\dots,y_m'\}$) we arrive at a contradiction to the minimality of $k$.
\end{proof}

\begin{remark}
One may ask whether there is a local-to-global argument taking as input $\mathrm{GL}_m(R_{\mathbb K})=\mathrm{GE}_m(R_{\mathbb K})$ for all $m\leq n$ and all fields $\mathbb K$ that leads to $\mathrm{GL}_n(R)=\mathrm{GE}_n(R)$. This is not the case. Indeed, the polynomial ring $R=\mathbb Z[t]$ satisfies the hypothesis for all $n$ and all $\mathbb K$, but the matrix
$$
\left(
\begin{array}{cc}
4&1+2t\\
1-2t&-t^2
\end{array}
\right)\in \mathrm{GL}_2(\mathbb Z[t])
$$
is not in $\mathrm{GE}_2(\mathbb Z[t])$ (\cite{cohngl}, p.30). In fact, $\mathrm{GL}_2(\mathbb Z[t])/\mathrm{GE}_2(\mathbb Z[t])$ is quite large (\cite{krsticmccool}). However, the above example goes away if we invert $t$, and \cite{abramenko} conjectures that 
$\mathrm{GL}_2(\mathbb Z[t,t^{-1}])=\mathrm{GE}_2(\mathbb Z[t,t^{-1}])$.
\end{remark}

\subsection{Chain complexes, cell decompositions, and Morse theory}
Recall that the cohomological dimension of the group $G$, denoted $\mathrm{cd}(G)$, is the minimal length of a free $\mathbb Z[G]$-resolution of $\mathbb Z$ (see \cite{brownbook}).   

\begin{theorem}
\label{cdim}
Assume that the group $G$ satisfies $\mathcal{H}_{n}$. 
\begin{enumerate}
\item[1.]
For every free $\mathbb Z[G]$-resolution $C_*\rightarrow\mathbb Z$ and each $0<k<\mathrm{cd}(G)$ we have
$$
\mathrm{rank}_{\mathbb Z[G]}(C_k)>n.
$$
\item[2.](Theorem \ref{cellbound} in the introduction)
Every aspherical cell complex with fundamental group $G$ has more than $n$ cells of each dimension $0 < k < \mathrm{cd} (G)$.
\end{enumerate}
\end{theorem}

\begin{proof}
First, we prove the algebraic part. 
Suppose rank$(C_k)\leqslant n$. Then the module of boundaries, $B_{k-1}:= \partial (C_k)\subset C_{k - 1}$ is an $n$-generated submodule of a free module. 
Since $C_{\ast}$ is a resolution, $\ker \left( C_{k - 1}
  \overset{\partial}{\rightarrow} C_{k - 2} \right) = \partial (C_k)$, so $C_{k - 1} / \partial (C_k)$ injects into $C_{k - 2}$, which is either a free module (if $k \geq 2$) or $\mathbb{Z}$ (if $k = 1$). In either case, $C_{k - 1} / \partial (C_k)$ is torsion-free as an abelian group. Hence, $B_{k-1}$ satisfies $(\star)$ and we conclude from Theorem \ref{freezmodules} that it is free. So, we obtain a new free resolution 
$$
0 \rightarrow B_{k-1} \rightarrow C_{k - 1}\rightarrow \ldots \rightarrow C_0 \rightarrow \mathbb{Z} \rightarrow 0 
$$
of length $k$. Thus $\mathrm{cd} (G) \leq k$. This proves 1.

Now, suppose $X$ is an aspherical cell complex with fundamental group $G$. Let $C_* (\widetilde{X} ; \mathbb{Z})$ be the cellular chain complex of the universal cover of $X$. The augmented complex $C_*(\widetilde X;\mathbb Z)\rightarrow \mathbb{Z}$ is a free $\mathbb{Z} [G]$ resolution of $\mathbb{Z}$ (\cite{brownbook}, Prop I. 4.1). Applying part 1 to this resolution gives part 2.
\end{proof}

\subsubsection*{Essential maps}
If $\mathbb{K}$ is a field, the same result is true with $\mathbb Z$ replaced by $\mathbb K$ and cohomological dimension by $\mathbb K$-cohomological dimension ($=$ minimal length of a free $\mathbb K[G]$-resolution of $\mathbb K$), and easier to prove as we don't need Theorem \ref{freezmodules}, but only Theorem \ref{freemodules}.2. In fact, in the setting of $\mathbb K$-cohomological dimension we have the following more general result suggested by a question of Gromov. 

Let $G$ be a classifying space for the group $G$. For a cell complex $X$, we say that a cellular map $X\rightarrow BG$ is {\it $d$-essential (with local coefficients)} if there is a $\mathbb K[G]$-module $V$ such that the induced map $H^{d}(BG;V)\rightarrow H^{d}(X;V)$ is non-zero. For the sake of brevity, we will omit ``with local coefficients'' from now on. We will call a $d$-manifold $X$ {\it essential} if the map $X\rightarrow B\pi_1(X)$ is $d$-essential.
\begin{example}
Any map of non-zero degree from a closed manifold to a closed aspherical manifold of dimension $d$ is $d$-essential. 
\end{example}
\begin{theorem}
\label{essential}
Assume that the group $G$ satisfies $\mathcal H_{n}$. If $X\rightarrow BG$ is a $d$-essential map, then $X$ has more than $n$ cells in each dimension $0<k<d$. 
\end{theorem}
\begin{proof}
Suppose not. Let $\hat X$ be the $G$-cover of $X$ induced by the map $f:X\rightarrow BG$. Then $\partial C_k(\hat X;\mathbb K)$ is an $n$-generated submodule of a free module, so it is a free module by Theorem \ref{freemodules}.2. Next, let $\iota$ denote the inclusion $\partial C_k(\hat X;\mathbb K)\hookrightarrow C_{k-1}(\hat X;\mathbb K)$. The image of $f_{k-1}\circ\iota:\partial C_k(\hat X;\mathbb K)\rightarrow C_{k-1}(EG;\mathbb K)$ is in the image of $\partial:C_k(EG;\mathbb K)\rightarrow C_{k-1}(EG;\mathbb K)$ since $f_*$ is a chain map $(f_{k-1}\partial=\partial f_k)$, so by lifting a free basis of $\partial C_k(\hat X;\mathbb K)$, we can construct a map $\lambda$ which makes the following diagram commute
$$
\begin{array}{cccccccc}
\dots\rightarrow & C_{k+1}(\hat X;\mathbb K)&\rightarrow &C_k(\hat X;\mathbb K)&\stackrel{\partial}\rightarrow &C_{k-1}(\hat X;\mathbb K)&\rightarrow\dots\rightarrow &C_0(\hat X;\mathbb K)\\
&\downarrow&&\downarrow\partial&&||\hspace{0.5cm}&&||\\
\dots\rightarrow & 0&\rightarrow &\partial C_k(\hat X;\mathbb K)&\stackrel{\iota}\hookrightarrow &C_{k-1}(\hat X;\mathbb K)&\rightarrow\dots\rightarrow &C_0(\hat X;\mathbb K)\\
&\downarrow&&\downarrow \lambda&&\downarrow f_{k-1}&&\downarrow\\
\dots\rightarrow & C_{k+1}(EG;\mathbb K)&\rightarrow &C_k(EG;\mathbb K)&\stackrel{\partial}\rightarrow &C_{k-1}(EG;\mathbb K)&\rightarrow\dots\rightarrow &C_0(EG;\mathbb K).
\end{array}
$$
So, we have constructed a chain map $f'_*:C_*(\hat X;\mathbb K)\rightarrow C_*(EG;\mathbb K)$ which agrees with $f$ on $C_{<k}$ and is zero on $C_{>k}$, In particular, since $d>k$, the chain map $f'_*$ induces the zero map on $H^d(-;V)$. 

On the other hand, we claim that $f'_*$ is chain homotopic to $f_*$.
The chain homotopy $h_*:C_{*}(\hat X;\mathbb K)\rightarrow C_{*+1}(EG;\mathbb K)$ is constructed as follows. 

For every $i<k$, we set $h_i=0$. 

In order to define $h_k$ we remark that the image of $f_k-f_k'$ is in the kernel of $\partial$ since $f_j=f'_j$  for $j<k$. Since $EG$ is contractible, this implies that $f_k-f_k'$ maps to the image of $\partial$. As $C_k(\hat X;
\mathbb K)$ is a free module, $f_k-f_k'$ lifts to a map $h_k:C_k(\hat X;\mathbb K)\rightarrow C_{k+1}(EG;\mathbb K)$ such that $\partial h_k=f_k-f_k'$. 

In order to define $h_{k+1}$, we observe that $\partial(f_{k+1}-h_k\partial)=0$. Indeed, 
\begin{eqnarray*}
\partial(f_{k+1}-h_k\partial)&=&\partial f_{k+1}-\partial h_k\partial \\
&=&f_k\partial-(f_k-f_k')\partial\\
&=&f_k'\partial\\
&=&\partial f'_{k+1}=0,
\end{eqnarray*}
where we have used the fact that $f_*$ and $f'_*$ are chain maps, and that $f'_*$ is zero on $C_{>k}$. Therefore, since $EG$ is contractible, $f_{k+1}-h_k\partial$ maps to the image of $\partial$. So, it has a lift $h_{k+1}:C_{k+1}(\hat X;\mathbb K)\rightarrow C_{k+2}(EG;\mathbb K)$. By construction $f_{k+1}=h_k\partial+\partial h_{k+1}$. 

In higher degrees we argue by induction, assume that for $i>k$ the maps $h_{\leq i}$ have been defined, and compute
\begin{eqnarray*}
\partial(f_{i+1}-h_i\partial)&=&\partial f_{i+1}-\partial h_i\partial\\
&=&f_i\partial-(f_i-h_{i-1}\partial)\partial\\
&=&0.
\end{eqnarray*}
So we can define $h_{i+1}$ as a lift of $f_{i+1}-h_i\partial$. Proceeding in this way, we obtain $h_*$ satisfying $f_i-f'_i=\partial h_i+h_{i-1}\partial$ in for all $i$. Therefore $f_*$ is chain homotopic to $f'_*$.

We conclude that $f_*$ induces the zero map on $H^d(-;V),$ contradicting the hypothesis that $f:X\rightarrow BG$ is $d$-essential. 
\end{proof}

Let us spell out a consequence of a special case of Theorem \ref{essential}, where we suppose that $X^d$ is a closed $d$-manifold.
Recall that a Morse function gives a cellular decomposition of a manifold with one cell of dimension $k$ for each point of index $k$ (\cite{thom}, or \cite{milnor} for a detailed proof). Therefore we deduce
{
\renewcommand{\thetheorem}{\ref{manifold}}
\begin{theorem}
Let $X^d$ be a $d$-manifold, $G$ a group that satisfies $\mathcal H_{n}$ and $BG$ its classifying space. If there is a continuous map $f:X^d\rightarrow BG$ with $f_*[X]\not=0$ in $H_d(BG;\mathbb K)$ then for each $0<k<d$, a Morse function on $X^d$ has at least $n+1$ critical points of index $k$.
\end{theorem}
\addtocounter{theorem}{-1}
}
\begin{remark}
Since the fundamental group of a hyperbolic manifold is not free, the case of critical points of index $1$ or $d-1$ follows from the aforementioned theorem of Arzhantseva, Gromov, Kapovich-Weidmann (\cite{arzhantseva,gromovexpanders,kapovichweidmann}). For hyperbolic manifolds of dimension three, a much better bound is given by \cite{bachman}.
\end{remark}

\subsection{Dimensions, few-relator groups, and $2$-complexes}
Concerning few-relator groups, we get the following.

{
\renewcommand{\thetheorem}{\ref{fewrelator}}
\begin{theorem}
An $n$-relator group satisfying $\mathcal{H}_{n}$ has cohomological dimension $\leq 2$.
\end{theorem}
\addtocounter{theorem}{-1}
}

\begin{proof}
An $n$-relator group is the fundamental group of an aspherical cell complex with $n$ $2$-cells (and cells of higher dimension). If such a group satisfies $\mathcal{H}_{n}$ then, by Theorem \ref{cdim}.2, it has cohomological dimension $\leq 2$.
\end{proof}

\begin{question}
\label{hypeg}
Does every $n$-relator group satisfying $\mathcal{H}_{n}$ have geometric dimension $\leq 2$? 
\end{question}
\begin{remark}
This question is an instance of a problem raised by Eilenberg and Ganea in \cite{eilenbergganea}, asking whether there is an example of a group whose geometric and cohomological dimensions differ: 

\begin{quote}
``{\it We do not know whether these exceptional cases are actually present. The inequality dim $\Pi<$ cat $\Pi$ (cases A and B) is equivalent with the assertion that $\Pi$ is one-dimensional but not free. The problem of the existence of such a group has equivalent formulations in terms of group extensions and also in terms of properties of the integral group ring $\Lambda=\mathbb Z[\Pi]$. Similarly, the inequality cat $\Pi<$ geom. dim $ \Pi$ (cases B and C) is related to properties of the ring $\Lambda$. For instance, if it can be shown that a direct summand of a free $\Lambda$-module is free, then the equality cat $\Pi$=geom. dim $\Pi$ follows.}'' (\cite{eilenbergganea}, p. 517-518).
\end{quote}
In \cite{eilenbergganea} ``dim'' is the cohomological dimension and there is also an intermediate dimension ``cat'' that was later shown to be equivalent to  cohomological dimension by Stallings \cite{stallings}. Cases A and B refer to the hypothetical situation ($1=\dim \Pi<$ cat $\Pi=2$) which is now known to not occur. Case C refers to the potential situation ($2=$ dim $\Pi=$ cat $\Pi<$ geom. dim $\Pi=3$). The conjecture that case C does not occur, either, is nowadays referred to as the Eilenberg-Ganea conjecture. A proof of the reduction of this conjecture to a question about group rings claimed in the last sentence of the quote may help with Question \ref{hypeg}, but we do not know how to establish such a reduction, even with the additional assumption $\mathrm{GL}_n(\Lambda)=\mathrm{GE}_n(\Lambda)$ for all $n$.
\end{remark}

Concerning the topology of presentation complexes, we get the following.

\begin{theorem}
Assume that the group $G$ satisfies $\mathcal{H}_{n}$.
\begin{enumerate}
\item[1.] 
Every presentation $2$-complex for $G$ with $n$ relations (or less) has a free $\pi_2$.
\item[2.](Theorem \ref{hypstandard} from introduction)
Assume further that the group $G$ has geometric dimension two, and let $Y$ be an aspherical $2$-complex with fundamental group $G$. Then every presentation $2$-complex with less than $n+1$ relations is standard, i.e. has the same homotopy type as $$Y \vee S^2 \vee \ldots \vee S^2.$$ 
\end{enumerate}
\end{theorem}

\begin{proof}
If $X$ is an $n$-relator presentation $2$-complex for such a group and $C_{\ast} = C_{\ast} (\tilde{X} ; \mathbb{Z})$ is the chain complex of its universal cover, then $\partial (C_2)$ is an $n$-generated submodule of $C_1$ satisfying $(\star)$. Hence it is a free module (by Theorem \ref{freezmodules}), and we have a splitting $C_2 \cong \ker \partial \oplus \partial (C_2)$. Since $X$ is a $2$-complex, the kernel is identified with $H_2(\tilde X)$ and, via the Hurewicz theorem, with $\pi_2(X)$. By projecting the generators of $C_2$ to $\pi_2 (X)$, we see that $\pi_2 (X)$ is generated by $n$ elements. Moreover, $\pi_2 (X)$ satisfies $(\star)$ (in fact, it is projective) so $\pi_2 (X)$ is free by Theorem \ref{freezmodules}, i.e. $\pi_2 (X) \cong\mathbb{Z} [G]^m$ for some $m$. This proves 1.
  
If $Y$ is an aspherical $2$-complex $Y$ with fundamental group $G$, we can construct a homotopy equivalence $f : Y \vee S_1^2 \vee \ldots \vee S^2_m \rightarrow X$ as follows. First, construct a map realizing the $\pi_1$-isomorphism $Y \rightarrow X$. This can be done since $Y$ is a $2$-complex. Second, let the $(S^2_i)_{1 \leqslant i \leqslant m}$ represent a $\mathbb{Z} [G]$-basis of $\pi_2  (X)$. Then the resulting map $f$ is a $\pi_1$-isomorphism and a homology isomorphism of universal covers. Since $f$ is a map of CW complexes, Corollary 4.33 of \cite{hatcherbook} implies that $f$ defines a homotopy equilvalence of universal covers. Therefore, $f$ induces an isomorphism on $\pi_i$ for all $i$. So, by Whitehead's theorem, $f$ is a homotopy equivalence.  This proves 2.
\end{proof}

\bibliography{trees}
\bibliographystyle{amsplain}
\end{document}